%% file: entropy.tex
\documentclass[11pt, a4paper]{article}
\usepackage{amsfonts,amsmath,amsthm,amssymb}
\usepackage[utf8]{inputenc}
\usepackage{xspace}
\begin{document}
\title{Diabolical entropy}
\date{}
\author{Neil Dobbs\thanks{N.D.\ was supported by the ERC Bridges grant} \and  Nicolae Mihalache\thanks{N.M.\ was supported by the ERC AG COMPAS grant, CNRS semester and ANR LAMBDA}}

\maketitle

\begin{flushright}
\textit{In memory of Tan Lei. \quad}
\medskip
\medskip
\end{flushright}

\begin{abstract}
    Milnor and Thurston's famous paper proved monotonicity of the topological entropy for the real quadratic family. Guckenheimer showed that it is Hölder continuous. We obtain a precise formula for the Hölder exponent at almost every quadratic parameter. 
    
Furthermore, the entropy of most parameters is proven to be in a set of Hausdorff dimension smaller than one, while most values of the entropy arise from a set of parameters of dimension smaller than one.
\end{abstract}

\input{macros.tex}

\section{Introduction}
\input{intro.tex}

\label{sectIntro}

\section{Preliminaries}
\label{sectPrel}
\input{sec2.tex}

%
%
\section{Proofs of Theorems \ref{thmHolderUnimodal} and \ref{thmParab}}
\label{sectMainproof}

\input{mainproof.tex}

\section{Parameter space}
\label{sectMeasure}
\input{measure.tex}

\section{Typical values}
\label{sectTyp}
\input{typical.tex}

\section{Uniform H\"older regularity}
\label{sectUniHol}
\input{uniHol.tex}

\bibliography{references}{}
\bibliographystyle{plain}

\end{document}

%% file: macros.tex
\newtheorem{lem}{Lemma}[section]
\newtheorem{thm}{Theorem}
\newtheorem{thmbody}[lem]{Theorem}
\newtheorem{conj}{Conjecture}
\newtheorem{mcor}[thm]{Corollary}
\newtheorem{cor}[lem]{Corollary}
\newtheorem{prop}[lem]{Proposition}
\newtheorem{cons}[lem]{Construction}
\newtheorem{defi}[lem]{Definition}
\newtheorem{fact}[lem]{Fact}
\newtheorem{rem}[lem]{Remark}

\def\ss{\subset}
\def\sm{\setminus}

\def\Orb{{\mathrm{Orb}}}
\def\eps{\varepsilon}
\def\del{\Delta}
\def\De{\Delta}
\def\ga{\gamma}
\def\ka{\kappa}
\def\Ga{\Gamma}
\def\be{\beta}
\def\om{\omega}
\def\Om{\Omega}
\def\al{\alpha}
\def\lam{\lambda}
\def\la{{\lambda}}
\def\si{{\sigma}}
\def\ph{\varphi}
\def\dd{\partial}
\newcommand{\de}{{\bf \delta}}
\newcommand{\un}{\underline}

\def\im{\Im}
\def\re{\Re}

\def\Cal#1{{\cal#1}}
\def\C{\mathbb C}
\def\N{\mathbb N}
\def\CC{\hat{\mathbb C}}
\def\D{\mathbb D}
\def\R{\mathbb R}

\def\sgn{{\mathrm{sgn}}}
\def\hd{{\mathrm{dim}_\mathrm{H}}}

\def\ip{{\ensuremath{\left[-2,\frac14 \right]}}}
\def\bip{{\ensuremath{\left[-2,\frac14 \right]}}}
\def\iv{{\ensuremath{[0,\ln2]}}}
\def\se{{\subseteq}}

\def\ol{\overline}
\def\ul{\underline}
\def\lla{{\underline\la}}
\def\ula{{\overline\la}}

\def\cC{{\Cal C}}
\def\M{{\Cal M}}
\def\H{{\Cal H}}
\def\J{{\Cal J}}
\def\F{{\Cal F}}
\def\JWR{{\Cal W}}
\def\E{{\Cal E}}
\def\O{{\Cal O}}
\def\V{{\Cal V}}
\def\SS{{\Cal S}}

\def\Epc{{\E}}

\def\ln{\log}

\def\acip{\mu_\mathrm{acip}}
\def\mme{\mu_\mathrm{max}}

\def\htop{h_\mathrm{top}}
\def\Hol{\mathrm{H\ddot{o}l}}

\def\Dist{\mathrm{Dist}}
\def\dist{\mathrm{dist}}

\def\pa{{\partial}}

\def\ND{\underline{\textbf{Neil says:}} \textbf}
\def\NM{\underline{\textbf{Nicu says:}} \textbf}

\def\wrt{{with respect to}}

\newcommand\ui[2]{\ensuremath{\left(#1, #2\right)}}

\newcommand\rthm[1]{{Theorem~\ref{thm#1}}}
\newcommand\rpro[1]{{Proposition~\ref{prop#1}}}
\newcommand\rprop[1]{{Proposition~\ref{prop#1}}}
\newcommand\rlem[1]{{Lemma~\ref{lem#1}}}
\newcommand\rcor[1]{{Corollary~\ref{cor#1}}}
\newcommand\rdef[1]{{Definition~\ref{defi#1}}}
\newcommand\rfact[1]{{Fact~\ref{fact#1}}}
\newcommand\rsec[1]{{\S\ref{sect#1}}}
\newcommand\rrem[1]{{Remark~\ref{rem#1}}}

\def\wg{well-rooted\xspace}

\def\mc{\simeq_\cdot} 
\def\ac{\simeq_+}
\def\Xi{\varphi}

\newcommand\matop[2]{\genfrac{}{}{0pt}{}{#1}{#2}}

%% file: intro.tex
This paper studies the regularity of the topological entropy of maps from the 
quadratic (or \emph{logistic}) family $f_a(x)=x^2+a$, $a\in\ip$. The fundaments concerning topological entropy for piecewise-monotone maps of the interval were developed by Misiurewicz and Szlenk \cite{MisSzl}.
They showed that the entropy $h(g)$ of a map $g$ is
    the exponential growth rate of the number of monotonic laps of iterates of $g$ and that, for smooth unimodal maps, the entropy varies continuously. 
    Writing $h(a)$ for $h(f_a)$, 
     Milnor and Thurston  \cite{MilThu}\footnote{Douady, Hubbard and Sullivan had proven that the number of periodic orbits of some fixed period is monotonically decreasing, which implies the monotonicity of entropy. This result was unpublished, a later version was published by Douady \cite{Dou}.}
     proved that $a \mapsto h(a)$ is a monotone function, a result recently generalised to the multimodal setting by Bruin and van Strien \cite{BruvS}. The function $h$ has range \iv, yet is locally constant on an open dense subset of $\ip$ and therefore quite irregular. 
     On the other hand, Guckenheimer \cite{Guc} proved that $a \mapsto h(a)$ is a H\"older continuous map. 
     Our principal result provides an exact formula for the H\"older exponent of $h$ at most parameters, given in terms of the value of $h$ itself and the Lyapunov exponent of the critical value. 

     \medskip

     In the presence of an attracting periodic orbit, the map is \emph{hyperbolic}, the non-wandering set is structurally stable and the entropy $h$ is locally constant. We denote the set of hyperbolic parameters in $\ip$ by $\H$. Hyperbolic parameters form an open, dense set \cite{GraSwi, Lyu}, so $h$ is locally constant on an open dense set. 
     The maximal set $\F$ on which $h$ is locally constant (or flat) strictly contains $\H$. Let us define $\V$ as the set of points $a$ at which $h$ is not locally constant at $a$ on either side of $a$. 
    By monotonicity and continuity of $h$, 
     $$ \V := \left\{ a \in \ip : \{a\} = h^{-1}(h(a)) \right\}.$$
     Non-renormalisable parameters with positive entropy are contained in $\V \cup \{-2\}$. 
As a by-product of Jakobson's theorem \cite{Jak}, $\V$ has positive measure.

The  \emph{lower Lyapunov exponent} $\lla(a)$ of $f_a(a)$ is defined by 
$$\lla(a):=\liminf_{n\to\infty}\frac{\ln|(f_a^n)'(a)|}n;$$
 the  \emph{upper Lyapunov exponent} $\ula(a)$ is defined with a $\limsup$ instead. 
 When $\lla(a)=\ula(a)$, we call the common value $\la(a)$ the \emph{(pointwise) Lyapunov exponent of the critical value}. 

\emph{Tsujii's weak regularity} condition \cite{Tsu}
\begin{equation} \label{equWR}
    \tag{WR}
    \lim_{\de \to 0^+} \liminf_{n \to \infty} \frac 1n \sum_{\matop{j=1}{\left|f_a^j(0)\right| \leq \de}}^{n}  \ln |f_a'(f_a^j(0))|=0
\end{equation}
will be rather important in this work. It says that the critical orbit may recur, but not too close too soon and not too often. Let
\begin{equation} \label{equJWRdef}
\JWR := \left\{a \in \ip : \la(a) >0 \mbox{ and } f_a \mbox{ verifies } \eqref{equWR} \right\}.
\end{equation}
We shall deduce (in \rprop{Quadwr}, based on work of Tsujii, Avila and Moreira and Lyubich) that  $\JWR \se \H^c$ has full measure in $\H^c$ and thus in $\V$.  



\begin{thm}
\label{thmHolder}
For every $a\in \V \cap \JWR$,  
\begin{equation}\label{equMagic}
    \lim_{t\to 0} \frac{\ln |h(a+t)-h(a)|}{\ln|t|} = \frac{h(a)}{\la(a)}.
\end{equation}
\end{thm}

Consequently the derivative satisfies $h'(a) =0$, for $a \in \V \cap \JWR$, 
if ${h(a)}/{\la(a)} >1$; 
if less than $1$, $|h'(a)| = \infty$.

Collet-Eckmann \cite{ColEck} parameters are those for which $\lla(a) >0$. 
Benedicks and Carleson \cite{BenCar}  proved that they form a positive measure set. Each parameter in $\JWR$ is Collet-Eckmann. 
The set $\ip \setminus (\V \cup \F)$ is necessarily countable; to one side of each parameter, the entropy is locally constant; each parameter in the set is either parabolic or preperiodic. For 
    preperiodic parameters,~\eqref{equMagic}  still holds 
    on the non-locally-constant side,
    see \rthm{HolderUnimodal}. 
For parabolic parameters, we obtain infinite flatness, see \rthm{Parab}.


We say that a real map $g:I \to \R$ is $(C,\be)$-H\"older continuous at $x \in I$ if  $C,\be>0$ and if, for all $y \in I$,
\begin{equation}\label{equHoldef}
    |g(y)-g(x)| \leq C |y - x|^\be.
\end{equation}
If the constants are not specified, we will say the map is $\be$-H\"older continuous  or H\"older continuous at $x$. We omit ``at $x$'' if~\eqref{equHoldef} holds at every $x \in I$. 
Let \label{pageH}
$$\Hol(g,x):=\sup\{\be > 0\ :\ g \text{ is }\be\text{-Hölder continuous at }x \}.$$
A weaker formulation of \rthm{Holder} would be that $\Hol(h,a) = \frac{h(a)}{\la(a)}$ for  every $a \in \V \cap \JWR$. 

From the proof, one can extrapolate that if $\la_n(a) = \frac1n \log |Df^n_a(a)|$ oscillates slowly but with large range as $n$ grows, then the limit~\eqref{equMagic} does not exist.
Thus $\la(a)$ in general does not exist, so one cannot do much better than \rthm{Holder} (which we show for all parameters $a \in \F^c$ subject to relative full-measure hypotheses: existence of $\lambda(a)$ and Tsujii's weak regularity condition). 

Isola and Politi \cite{IsoPol} performed numerical experiments and some analysis on the regularity of $h$. They suggested that its local H\"older exponent at some parameter $a$, as a function of a number $\tau(a)$ related to the kneading determinant,  is its value $h(a)$. 
To our knowledge, \rthm{Holder} is the first non-experimental result concerning the regularity of $a \mapsto h(a)$ since the works of Guckenheimer and Milnor and Thurston.

\subsection{Uniform estimates and dimension}
\label{sectIUE}
The orbit of the critical value determines to a large extent the ergodic properties of the map. Conversely, the critical value is often typical with respect to a measure, so knowing the ergodic properties, one can sometimes determine properties of the post-critical orbit. These will permit us to deduce dimension estimates in Theorem~\ref{thm:main}.

To avoid confusion with the pointwise Lyapunov exponent of the critical value, given a map $f$, we shall denote the Lyapunov exponent of an $f$-invariant probability  measure $\mu$ by
$$
\chi(\mu) = \int \log |f'|\, d\mu.$$
There are two rather special types of such measures,
\begin{description}
    \item[$\mme$ ] : the measure of maximal entropy;
    \item[$\acip$] \hspace{2mm}: an absolutely-continuous invariant probability  (\emph{acip}).
        \end{description}
        The Feigenbaum (-Coullet-Tresser) parameter $a_F$ is the leftmost parameter $a$ satisfying $h(a)=0$.
%
For a quadratic map $f_a$, $a \in [-2,a_F)$, the measure of maximal entropy always exists, is unique and its metric entropy $h(\mme)$ equals $h(a)$ \cite{Rai} 
            (this function overloading of $h$, so it can take as an argument either a parameter or an invariant measure, ought not cause confusion). There may or may not be an acip, but if there is, it is unique and \cite{BloLyu, Led} 
            $$\chi(\acip) = h(\acip) >0.$$  
    We shall use $\mme^a, \acip^a$ to indicate dependence on $f_a$. 

    Let us define 
    \begin{eqnarray}
        \hat X &:=&\{a\in \JWR  : \la(a) = \chi(\acip^a)\},\\ \label{equXhatdef}
        X &:=& \{a\in \V \cap \JWR  : \la(a) = \chi(\acip^a)\},\\ \label{equXdef}
        Y &:=& \{a\in \V \cap \JWR  : \la(a) = \chi(\mme^a)\}. \label{equYdef}
    \end{eqnarray}

    $\hat X$ will have full measure in $\H^c$  and
    $X = \hat X \cap \V$ will have full measure in $\F^c$ (\rprop{Quadwr}).
 It follows from \cite{San, Bru} (see Proposition~\ref{prop:tentwr}) that $h(Y)$ has full measure in $[0, \log 2]$. 

 \begin{defi}
\label{defUni}
We say that a continuous map $g:I \to I$ defined on a compact interval $I$ is \emph{unimodal} if $g(\pa I) \ss \pa I$ and $g$ has exactly one turning point $c$ with $c \in I \sm \pa I$. We say $g$ is a \emph{smooth unimodal map} if, moreover, $g$ is continuously differentiable and $c$ is the unique (\emph{critical}) point satisfying $g'(c)=0$. The critical point is \emph{non-degenerate} if $g''(c) \ne 0$. 
\end{defi}
\begin{defi}
    A map $g : I \to I$  is \emph{non-degenerate S-unimodal} if it is a $\cC^2$ smooth unimodal map with non-degenerate critical point $c$,  $|g'|^{-1/2}$ is convex on each component of $I\setminus \{c\}$ and $|g'| >1$ on $\partial I$. 
\end{defi}

The convexity condition is equivalent \cite{NowSan}, for $\cC^3$ maps, to having non-positive Schwarzian derivative, while strict convexity corresponds to negative Schwarzian derivative. Quadratic maps have negative Schwarzian derivative.

Except for some special cases \cite{DobMME}, if $g$ is  S-unimodal with an acip $\acip$ of positive entropy, then $h(\acip) < h(\mme) = \htop(g)$.  
In the context of the quadratic family, we obtain  the following result of independent interest (whose likelihood was suggested by Bruin \cite{Bru}), improving on  the work of the first author \cite{DobMME}. It depends on a recent result of Inou \cite{Inou} on extensions of conjugacies of polynomial-like maps and is proven in~\S\ref{S:Vmme}. 
    \begin{thm} \label{thm:Zdun}
        If $f_a : x \mapsto x^2 + a$, 
          $\mme^a = \acip^a$ if and only if $a = -2$. 
    \end{thm}
    It was known already to von Neumann and Ulam \cite{vNU} that $\acip^{-2}$ is the pullback of Lebesgue measure by the smooth conjugacy to the tent map $x \mapsto 1-2|x|$. It follows that $h(\acip^{-2}) = \log 2$ and $\acip^{-2} = \mme^{-2}$. 
    In~\S\ref{S:Vmme} we show that absolute continuity of $\mme^a$ implies $a = -2$.

    For $a \ne -2$,  
    $h(\acip^a) < h(a)$ (when $\acip^a$ exists). 
    For $a =-2$, one can verify that $\la(a) = 2\chi(\mme^a) = 2\chi(\acip^a).$ Thus $-2 \notin  \hat X  \cup Y$. 
    On $\hat X$, we deduce 
    \begin{equation}\label{eqn:chihrel}
        \la(a) = \chi(\acip^a) = h(\acip^a) < h(\mme^a) = h(a).
        \end{equation}
        For $a \in \H$, $\la(a) < h(a)$. 
        On $\F$, $h' = 0$ by definition. 
        Applying \rthm{Holder} on $X$, a full measure subset of $\F^c$,  we obtain the following. 
\begin{mcor}
\label{corDS}
For almost every $a \in \ip$, $\la(a) < h(a)$ and $h'(a)=0$.
\end{mcor}
Thus  $h$ is absolutely singular. A \emph{devil's staircase} function can be defined as a non-constant, continuous, monotone function which has derivative $0$ almost everywhere and is locally constant on an open dense set. 
In particular, $h$ is a devil's staircase function. 
We shall improve on this, obtaining uniform dimension estimates away from $\{-2,a_F\}$. 

Recent tools from \cite{DobTod} allow us to prove uniformity in~\eqref{eqn:chihrel} and to prove continuous dependence of $\chi(\mme^a)$ on $a$ (Theorem~\ref{thm:mmecns}). We obtain the following,
where $\hd(A)$ denotes the Hausdorff dimension of a set $A$.
\begin{thm} \label{thm:main}
\label{thmImage}
For every $\eps>0$,
\begin{eqnarray*}
    \hd\left(h\left(X \cap [-2+ \eps, a_F - \eps] \right)\right) &<& 1,\\
    \hd\left(Y \cap [-2+ \eps, a_F - \eps] \right) &<& 1.
\end{eqnarray*}
\end{thm}
This is much stronger than just absolute singularity of $h$.  Neglecting a neighbourhood of two points, full measure (note that $\hd(h(\F))=0$) gets mapped to dimension strictly less than one and dimension strictly less than one gets mapped to full measure. 
We shall also show that removing a neighbourhood of $a=-2$ is necessary, for otherwise the above dimensions tend to $1$, see \rthm{HDim1}. Numerically straightforward estimates indicate that the uniformity estimates extend to a  neighbourhood of $a_F$; however, the method known to the authors requires estimating multipliers in a renormalisation limit, whose rigorous proof would be inappropriately lengthy. 

It is possible to find positive measure subsets of $X$ on which $h(a)/\la(a)$ is arbitrarily large. This leads to:
\begin{thm}
\label{thmSuperHolder}
    Given $\eps>0$, there is a positive measure subset of $\F^c$ whose image under $h$ has dimension at most $\eps$. 
\end{thm}

\subsection{Renormalisation}
\begin{defi}
\label{defRen}
We say that a unimodal map $g:I \to I$ is \emph{renormalisable} of period $n \geq 2$ if there is an open interval $J \ss I$ containing its turning point such that its images $g^i(J)$ with $i=0,1,\ldots,n-1$ are disjoint and the restriction of $g^n$ to $\ol J$ is unimodal. The interval $J$ is called a \emph{restrictive interval}.
    A period-two renormalisation is called a \emph{Feigenbaum} renormalisation. 
 \end{defi}
 The flat set $\F \subset \ip$ is the union of the interval $\left(a_F, \frac 14 \right]$ and the interior of all non-Feigenbaum renormalisation windows.
 \begin{defi} \label{defiPreC}
     A unimodal map $g$ is said to be \emph{pre-Chebyshev} if: a) $g$ is exactly $m$ times renormalisable, for some $m\geq 0$, and each renormalisation is of period two; b) if $J$ is the restrictive interval for the $m^{\mathrm{th}}$ renormalisation, $g^{2^m}_{|J} : J \to J$  is smoothly conjugate on $J$ to $x \mapsto 1 - 2|x|$ on $(-1,1)$.
 \end{defi}

\subsection{Unimodal families}
It is reasonable to ask for local results for more general unimodal families than the quadratic family. One must pass from the small scale in parameter space to the large scale in phase space to obtain such results. Tsujii \cite{Tsu} provided the tools to do this and we work within his setting. More general results are possible, but would require for example developing Tsujii-Benedicks-Carleson type results for unimodal maps with degenerate critical points or, for example, using Sands' techniques \cite{SanMMR} to deal with a lack of transversality. 

\medskip

\noindent
\textbf{Well-rooted family.}
    Let $I$ denote the compact interval $[-1,1]$. Consider
    $G : I\times [0,1] \to I$  of class $\cC^2$ with $g_t = G(\cdot, t)$, so $G$ defines a one-parameter family of interval maps $g_t : I \to I$. We denote by $\pa_1 G$, $\pa_2 G$ the partial derivatives with respect to the first and second variables. 
    Suppose each $g_t$  is a smooth unimodal map with  critical point  situated  at $0$. 
    For $g_0$ we impose that
    \begin{itemize}
        \item
            $g_0$ is a non-degenerate S-unimodal map;
        \item
            the Lyapunov exponent $\lambda_0$ of the critical value of $g_0$ exists and $\lambda_0 > 0$;
        \item
            Tsujii's weak regularity condition holds:
        \begin{equation} \label{equWRtake2}
        \lim_{\de \to 0^+} \liminf_{n \to \infty} \frac 1n \sum_{\matop{j=1}{\left| g_0^j(0) \right| \leq \de}}^{n} \ln |g_0'(g_0^j(0))|=0 ; 
        \end{equation}
        \item
            the following transversality condition holds at $t=0$:
    $$
    \sum_{j=0}^\infty \frac{\partial_2 G(g_t^j(0), t)}{(g^j_t)'(g_t(0))} \ne 0.$$
    \end{itemize}
    Such a $g_0$ satisfies the backward Collet-Eckmann condition and all periodic points of $g_0$ are hyperbolic repelling \cite[Theorem~A]{NowSan}. For such $g_0$, the topological entropy is positive. 
    \begin{defi} \label{defiWellG}
        We call $G$, as above, a \emph{\wg unimodal family}. 
    \end{defi}
    Let us simplify the notation by $h(t):=\htop(g_t)$. 
    We have the local, general form of \rthm{Holder}. It is unclear whether the monotonicity hypothesis is strictly necessary.
    \begin{thm} \label{thmHolderUnimodal} Let $G$ be a \wg family and assume that $h$ is monotone and that $h$ is not locally constant at $t=0$. Then
        $$
        \lim_{t\to 0^+} \frac{\ln | h(t)-h(0)|}{\ln|t|} = \frac{h(0)}{\lambda_0}.
        $$
    \end{thm}
    \begin{cor}
    	\label{corThmHolder}
        \rthm{Holder} holds. 
    \end{cor}
    \begin{proof}
        We need to check that for $a \in \JWR$, $f_a$ verifies the conditions of $g_0$. 
        Quadratic maps are non-degenerate S-unimodal maps. 
        The definition of $\JWR$ covers the next two conditions. 
         Levin \cite{Lev} proved transversality holds if
        $$\sum_{k \geq 0} \frac 1{|(f_a^k)'(a)|} < \infty.$$
        The above sum is finite for all $a \in \JWR$ since $\la(a) >0$, so transversality does indeed hold. 
        Thus one can apply \rthm{HolderUnimodal}. 
    \end{proof}

    When $g_0$ is preperiodic, one could refine the estimates to obtain actual H\"older continuity of $h$ at $0$ with exponent $\Hol(h,0)$. 
    The argument follows the proof of \rthm{HolderUnimodal}, noting that for every time $n$, one can pass with bounded distortion from small scale in parameter space to large scale in phase space (see for example 
    \cite{SanMMR}).  

At the boundary of a flat interval for $h$, there is either a pre-Chebyshev or a (primitive) parabolic parameter. A pre-Chebyshev renormalises to a full-branched map, thus it is preperiodic and fully described by the previous remark. At parabolic parameters $h$ is very flat: 

\begin{thm}
\label{thmParab}
Let $G$ be a $\cC^2$ family of $\cC^2$ unimodal maps so that $g_0$ has a parabolic periodic orbit which attracts, but is disjoint from, its critical orbit. Then
$$\Hol(h,0) = \infty.$$
\end{thm}
For S-unimodal maps, the existence of a parabolic periodic point implies that the critical point is contained in the immediate basin of attraction. In \rsec{Parab}, we shall present more precise estimates for the flatness of entropy at parabolic parameters.

\subsection{H\"older estimates}

Guckenheimer \cite{Guc}
assumes that the entropy is bounded away from zero to obtain uniform H\"older regularity for unimodal families. We extend this result to the full quadratic family. 
\begin{thm}[{\cite{Guc}}]
\label{thmUniHol}
For the quadratic family, the function $h$ is H\"older continuous.
\end{thm}
The exponent $\beta$ of H\"older continuity which is obtained will depend on a bound from renormalisation theory. 
Numerical and heuristic estimates 
lead us to
 formulate the following, recalling $a_F = \inf\{a : h(a) = 0\}$.
\begin{conj} \label{conjUniHol}
    There exists $\be>\frac 12$ such that, on any compact subinterval of $(-2, a_F)$, $h$ is $\be$-H\"older continuous.
\end{conj}
Finding the optimal $\be$ appears difficult. 
As estimated in \cite{Guc},
$\Hol(h,-2) = 1/2$, while 
straightforward computation gives
$$\Hol(h,a_F) = \frac{\ln 2}{\ln \de^*}=0.449\ldots,$$
where $\de^*=4.669\ldots$ is the first Feigenbaum constant. In particular, the exponent is worse at the boundary of $[-2,a_F]$.

\subsection{Further remarks}

Recent results \cite{BanRug, CarTio} consider the regularity of the entropy in families of dynamical systems with holes. That is, for a fixed map on the circle or the interval, they remove an interval and all its preimages from the phase space, thus modelling physical systems which leak mass. By varying the size and the position of the initial gap, they show that the entropy is H\"older continuous. In both cases, the regularity constant involves the value of the entropy itself.

Dudko and Schleicher and Tiozzo \cite{DudSch, Tio} have developed  the theory of \emph{core entropy} of complex quadratic polynomials, 
as introduced by Thurston.  They proved continuous dependence of core entropy on the parameter. Their work does not apply to regularity of $h$. 

\subsection{Structure}

Unimodal maps are canonically semi-conjugate to tent maps with the same entropy. Facts concerning tent maps, together with some further preliminaries, are presented in the following section. 

The proof of \rthm{HolderUnimodal} will depend on studying the movement of critical orbits as one varies the parameter, both for  unimodal families of smooth maps and the family of tent maps. The speed depends on the Lyapunov exponents of the image of the respective turning points. For the tent map, the Lyapunov exponent is just the entropy, and the parameter change is comparable to the entropy change. Combining these elements forms the backbone of the proof, presented in \rsec{HU}. Theorem~\ref{thmHolder} follows from \rthm{HolderUnimodal}, as already noted. In \rsec{PAR}, we provide refined estimates near parabolic parameters which imply \rthm{Parab}.

In \rsec{Measure}, we show that certain sets of parameters have large measure. 

Theorem~\ref{thm:Zdun} is proven in~\S\ref{S:Vmme}. More generally, 
in~\rsec{Typ}, we study the relation between $h(a)$ and $\la(a)$.   For $a \in X\cup Y$,  $h(a) \ne \la(a)$. We use Theorem~\ref{thm:Zdun} and thermodynamic formalism for families of unimodal maps to obtain uniformity in $h(a) \ne \la(a)$ on  $(X \cup Y)\cap [-2+\eps, a_F-\eps]$. 
Using \rthm{Holder} and a simple dimension lemma, we prove Theorems~\ref{thm:main} and~\ref{thmSuperHolder}.

In the final section we prove \rthm{UniHol}, giving uniform H\"older continuity for the entire quadratic family.

\subsection{Acknowledgments}
The authors thank Magnus Aspenberg, Viviane Baladi, Michael Benedicks, Davoud Cheraghi, Jean-Pierre Eckmann, Jacek Graczyk and Masato Tsujii for helpful comments and conversations.

%% file: sec2.tex
\def\ln{\log}

\subsection{Tent maps, kneading itineraries and conjugacies}
\label{sectTent}
Let us introduce the family of (symmetric) tent maps 
$$T_b:x \mapsto 1 - b|x|$$
for $b \in (1,2]$ and $x\in\R$. The orientation-preserving fixed point is $\frac{-1}{b-1}$. 
$T_b$ restricted to the interval
$$
    \hat I_b := \left[\frac{-1}{b-1}, \frac{1}{b-1}\right]$$ 
    is unimodal. 
If $b > \sqrt{2}$ then $T_b$, restricted to $\hat I_b$, is not renormalisable, while if $b \in (1, \sqrt{2}]$, it is (Feigenbaum) renormalisable of period two. 
The tent map  $T_b$ has slope $\pm b$ and topological entropy $\log b$. We call a tent map \emph{periodic} if the orbit of the turning point $0$ is periodic. The period is necessarily at least $3$. 
 
 Given a unimodal map $g$, smooth on the complement of its turning point $c$, let 
 $$
 K(g):=k_1 k_2 \ldots, \text{ where }k_i=\sgn(g'(g^i(c))\in\{0, \pm 1\}.$$
 We call $K(g)$ the \emph{kneading itinerary} of $g$. 
 For $i\geq 0$, let $d_i:=\prod_{j=1}^i k_j$, thus $d_0=1$, $d_1=\sgn(g'(g(c)))$ and $d_i = \sgn((g^i)'(g(c)))$. The kneading determinant $D_g$ is defined by
$$D_g(t):=1+d_1 t + d_2 t^2 + \cdots.$$
By Milnor and Thurston \cite{MilThu},  
$$s:=\inf\{t > 0\ :\ D_g(t)=0\}\in(0,1)$$ 
if and only if $$\htop(g) > 0,$$
and in this case
$$\htop(g) = -\ln s.$$
In particular, the kneading itinerary determines the entropy. 

Following the notation of \cite{BruMis}, for $n \geq 0$, let
\begin{equation}
\label{equDefXi}
	\Xi_{n}(b) := T^{n}_b(1).
\end{equation}
Observe that $$\Xi_i(b)=1-b|\Xi_{i-1}(b)|=1-b\,\sgn(\Xi_{i-1}(b))\Xi_{i-1}(b)=1+k_i b\,\Xi_{i-1}(b).$$
Note that $\Xi_{i-1}(b) \in \hat I_b$, so  if $|\Xi'_{i-1}(b)|$ is sufficiently large, then 
\begin{equation}\label{equXiratio}
    1+\frac{b-1}2 < \left| \frac{\Xi'_i(b)}{\Xi_{i-1}'(b)}\right| < 1 + 2(b - 1).
\end{equation}
\begin{lem} \label{lemBruMisLem}
    Given $b_*>1$, there are constants $C,N\geq1$ such that, if $b \geq b_*$ and $n \geq N$, the derivative $\Xi_n'(b)$ exists unless
    $$ \Xi_j(b) = 0$$
    for some $j$ with  $1 \leq j < n$. If $\Xi_n'$ exists on an interval $V \subset [b_*, 2]$ then
    \begin{equation}\label{equXibnCbn}
    C^{-1}b_0^n \leq |\Xi_n'(b)| \leq Cb_0^n
\end{equation}
    for all $b,b_0 \in V$. 
\end{lem}
\begin{proof}
    Brucks and Misiurewicz \cite{BruMis} proved this lemma for $b_* = \sqrt{2}$. 
We shall consider the case $b \in (1, \sqrt{2}]$. On every interval on which $\Xi_j(b) \neq 0$ for all $1 \leq j < n$, $\Xi_n$ is a polynomial of degree $n$, thus continuously differentiable.

For $m \geq 0$ and $b_m = 2^{2^{-m}}$, one has $\Xi_j(b_m) \ne 0$ for any $j \geq 1$. The derivative is continuous at each $b_m$ and they form (for bounded $m$) a discrete set of parameters, so we can omit them from our considerations.

Let $m \geq 1$ be such that $b^{2^m} \in (\sqrt{2}, 2)$ and 
set $\hat b = b^{2^m}$. Note that $T_b$ is $m$ times renormalisable. Let $A_b$ denote the affine rescaling which maps the closure of the restrictive interval for $T_b^{2^m}$ to the interval $\hat I_{\hat b}$, so
    $$
    A_b \circ T_b^{2^m} = T_{\hat b} \circ A_b$$
    on $A_b^{-1}(\hat I_{\hat b})$. 
    Since $A_b(0) = 0$, $A_b$ is just multiplication by some number $a_b$ bounded away from $0$. The boundaries of the restrictive interval and of $\hat I_{\hat b}$ are preperiodic points with smooth continuations, so $b \mapsto a_b$ is smooth as long as $T_b$ remains $m$-times renormalisable.
    Differentiating 
    $$\Xi_{2^mk-1}(b) = a_b^{-1}\Xi_{k-1}(b^{2^m}),$$ 
    one obtains
    $$
    \Xi'_{2^mk -1}(b) =\frac{ 2^mb^{2^m-1}}{a_b}\Xi'_{k-1}(\hat b) - \Xi_{k-1}(\hat b) \frac{a_b'}{a_b^2}.$$
    By \cite{BruMis}, if $\Xi_j(\hat b) \ne 0$ for $j=1,\ldots, k-2$, then $\Xi_{k-1}'(\hat b)$ exists and is comparable with $\hat b^{k-1}$, so this term dominates for large $k$, $a_b$ and $a_b'$ being bounded functions. 
    Bound~\eqref{equXibnCbn} thus holds for $n= 2^mk-1$, for all large $k$, for $b = b_0$. 
    Applying~\eqref{equXiratio} allows one to interpolate, giving~\eqref{equXibnCbn} for all large $n$, for $b=b_0$.

    Let $V \subset [b_*, 2]$ be an interval on which $\Xi_n'$ exists. On $V$, $|\Xi_n'| \geq C^{-1} b_*^n$. 
    For $b, b_0 \in V$, we deduce 
    $$|b-b_0| \leq  C b_*^{-n}.$$
    Therefore, 
    $$\frac{b^n}{b_0^n} = \left(1 + \frac{b - b_0}{b_0}\right)^n \leq C_1.$$ 
    Redefining $C$,~\eqref{equXibnCbn} is proven. 
    \end{proof}

For a quadratic map $f_a : x \mapsto x^2 +a$, $a \in \ip$, the fixed points solve $x^2 +a = x$, so $f_a$ restricted to the interval 
\label{pageIa}
$$
I_a := \left[\frac{-1 - \sqrt{1-4a}}{2}, \frac{1 + \sqrt{1-4a}}{2}\right]$$
is unimodal.

Given any unimodal map $g$, 
there is a canonical semi-conjugacy $\phi$, 
    $$
    \phi \circ g 
    = 
    T_b \circ \phi,
    $$
    between $g$ and $T_b$, where $b = \exp(\htop(g))$. See \cite[Chapter~II]{deMvS} as a general reference for this and the following statements. 
    S-unimodal maps do not have wandering intervals, and any non-repelling periodic orbit (or non-degenerate interval of periodic points) has the critical point in its immediate basin. If $T_b$ is not periodic and $g$ is S-unimodal, then $\phi$ is a conjugacy. If $T_b$ is periodic, the interior of the preimage of  the turning point by the conjugacy, \emph{i.e.}\ of $\phi^{-1}(0)$, is a restrictive interval for $g$. Since the period of $T_b$ is greater than two,  $g$ is renormalisable of some period other than two. 

    If $G$ is a \wg family, the entropy $\htop(g_t)$ is a continuous function of $t$ \cite[Theorem~2]{MisJumps}. Thus $\htop([g_{t_0}, g_{t_1}])$ contains $[\htop(g_{t_0}), \htop(g_{t_1})]$. 
    If $\log b \in (\htop(g_{t_0}), \htop(g_{t_1}))$, there exists $t \in (t_0, t_1)$ for which $g_t$ and $T_b$ have the same kneading itinerary, that is, 
\begin{equation}
\label{equKnead}
    K(g_t) = K(T_b),
\end{equation}
which implies that $g_t$ and $T_b$ are semi-conjugate and have the same topological entropy $\log b$. If $T_b$ is periodic and $g_t$  satisfies~\eqref{equKnead}, the critical orbit of $g_t$ is periodic and, therefore, super-attracting.

     If $g$ is a S-unimodal with 
     $\htop(g) \leq \frac{\ln 2}{2^m}$, then $g$ is $m$ times Feigenbaum renormalisable (that is, with period two) to a unimodal map $\hat g$ with 
\begin{equation}
\label{equRenHTop}
	\htop(\hat g) = 2^m \htop(g).
\end{equation}

\subsection{Notations and conventions}
\label{sectNot}

Given two  non-negative expressions $A(\cdot)$ and $B(\cdot)$, we say that $A$ dominates $B$ and write 
$$A \gtrsim B,$$
if there exists a constant $C>0$ such that
$$A \geq C B.$$
We say that $A$ and $B$ are comparable and write 
$$A \mc B$$
if $A \gtrsim B$ and $B \gtrsim A$.
As entropy is a logarithmic expression, we will need an additive equivalent. We say
$$A \ac B$$
if $e^A \mc e^B$. 

We write $(x,y)$ for the open interval bounded by $x,y$ irrespective of which of $x,y$ is greater, and similarly for half-open or closed intervals.

%% file: mainproof.tex
Following Tsujii \cite{Tsu}, we denote the \emph{distortion} of a smooth map $\psi$ on an interval $J$ by
$$\Dist(\psi,J):=\sup_{x,y\in J}\log \left| \frac{\psi'(x)}{\psi'(y)} \right| \in \ol\R_+.$$

Let $G : I \times [0,1] \to I, \quad G : (x,t) \mapsto g_t(x)$ be a $\cC^2$ family of $\cC^2$ unimodal maps with  each $g_t$ having its critical point at $0$. 
Define, for $n \geq 0$, 
    $$  
    \xi_n(t) := g_t^n(g_t(0)).$$
    These functions describe the displacement of the critical orbit as one varies the parameter. It is straightforward to check that 
\begin{equation}
\label{equQn}
	\frac {\xi_n'(t)}{(g_t^n)'(g_t(0))} = \sum_{j=0}^{n} \frac{\partial_2 G(g_t^j(0), t)}{(g^j_t)'(g_t(0))}.
\end{equation}
For the remainder of the section, set $h(t) := \htop(g_t)$. 

\subsection{Proof of \rthm{HolderUnimodal}} \label{sectHU}

Let $G : I \times [0,1] \to I, \quad G : (x,t) \mapsto g_t(x)$ be a \wg unimodal family as per \rdef{WellG}.
The unique critical point of each $g_t$ is at $0$. Let us denote $g:=g_0$. The Lyapunov exponent of its critical value 
$$v:=g(0)$$
exists and is denoted $\lambda_0$. Let $b_0 := \exp(h(0)).$

    To prove \rthm{HolderUnimodal}, the first step is to find an increasing sequence of times $(k_n)_n$ and decreasing parameter intervals $\om_n\ni 0$ for which $\xi_{k_n}$ maps $\om_n$ to the large scale in phase space with bounded distortion. 
    
    Next, there are topologically defined points (safe elements for tent maps and their preimages under the (semi-)conjugacy) contained in $\xi_{k_n}(\om_n)$. 
    One can catch these topologically-defined points, that is, there is a parameter $t_n \in \om_n$ which gets mapped by $\xi_{k_n}$ onto such a point. 
    
    The logarithmic size of $t_n$ is roughly $-k_n\la_0$. Doing the same thing for tent maps, for the corresponding tent parameter $b_n = \exp(h(t_n))$, the logarithmic size of $|b_n-b_0|$ is roughly $-k_n \log b_0$, and $|b_n - b_0|$ is comparable with $|\log b_n - \log b_0|$, the difference in the entropies. One obtains an estimate for the entropy at each $t_n$, which gives the result. 
    
    The details constitute the remainder of this subsection. 

\begin{lem}[{\cite{Tsu}}]
\label{factLS}
Let  $\delta,\eps >0$. There exist $r_0, m_0>0$, a strictly increasing sequence $(k_n)_{n\geq 0}$ of positive integers and a decreasing sequence of intervals $\om_n = [0, s_n)$ such that, for all $n\geq 0$, 
$$\frac{k_{n+1}}{k_n} \leq 1+\delta,$$
$$ |\xi_{k_n}(\om_n)| > r_0,$$
\begin{equation}
\label{equBD}
\Dist(\xi_{j}, \om_n) \leq 1 \quad \mbox{for $m_0 \leq j \leq k_n$},
\end{equation}
and,  for all $t \in \om_n$, 
\begin{equation}
\label{equCE}
\lambda_0-\eps \leq \frac{1}{k_n} \log |\xi'_{k_n}(t)| \leq \lambda_0+\eps.
\end{equation}
\end{lem}
\begin{proof}
    This was essentially shown by Tsujii \cite{Tsu}. Since $\lambda_0$ exists,  for any $\kappa>0$, for all large $m$, 
    \begin{equation}\label{eq:lyapeps}
        \lambda_0 -\kappa \leq \frac1m \log |(g^m)'(v)| \leq \lambda_0 + \kappa.
    \end{equation}
This implies 
$$\liminf_{n \to \infty} \frac 1n \log |g'(g^n(v))|=0,$$
which is condition (W) in \cite{Tsu}.    
    We apply \cite[Lemma~4.5]{Tsu} [our $g_t$ is Tsujii's $f_t$]. Let $\zeta >0$ be small. Then  there exists $\rho >0$ and a number $a^+(m)$ [the extra variables in Tsujii's $a^+(\cdot,\cdot;\cdot)$ are constant, for our purposes]  such that for each sufficiently large $L$, there exists $m = m(L)$ with 
    \begin{equation}\label{eq:zetaeps}
        L
        \leq  \log |(g^m)'(v)| \leq L(1+\zeta)
    \end{equation}
    and 
    \begin{equation}\label{eq:rhoeps}
        |(g^m)'(v)| \cdot a^+(m) > \rho.
    \end{equation}
    From~\eqref{eq:lyapeps} and~\eqref{eq:zetaeps} with $\kappa,\zeta$ sufficiently small, we obtain a strictly increasing sequence $(k_n)$ with $k_{n+1}/k_n \leq 1+\delta$ for which~\eqref{eq:rhoeps} holds for each $m = k_n$.  

    Now we apply~\cite[Lemma~5.2]{Tsu}. Inequality~(5.1) there is verified since $g$ is Collet-Eckmann.
    We obtain, for some constant $l_1>1$, for each $n$,
    $$
    l_1^{-1} < \frac{|\xi'_{k_n}(0)|}{|(g^{k_n})'(v)|} < l_1$$
    and 
    $$
    \Dist(\xi_{j},\om_n) \leq 1,$$
    for $m_0 \leq j \leq k_n$ (for some fixed $m_0$), 
on an interval $\om_n$ of length at least $l_1^{-1}\cdot  a^+(k_n)$. Thus 
    $$|\xi_{k_n}(\om_n)| > l_1^{-1}\cdot  a^+(k_n) \cdot l_1^{-1} |(g^{k_n})'(v)|.$$
    The result follows, taking $r_0 := l_1^{-2}\rho$. 
\end{proof}

\begin{lem}
\label{lemNCZ}
In the setting of Lemma~\ref{factLS}, for $\eps,\de>0$ small enough, for $n$ large enough and all $m<k_n$, the image $\xi_m(\om_n)$ does not contain the critical point $0$.
\end{lem}
\begin{proof}
Since the $\cC^1$-norm of $G$ is bounded and 
$$\xi'_{m+1}(t) = \xi'_m(t) g_t'(\xi_m(t)) + \pa_2 G(\xi_m(t), t),$$ 
$|\xi'_{m}|$ cannot grow too fast. 
Moreover,
if $\xi_m(t) = 0$, then 
$$|\xi'_{m+1}(t)| = |\pa_2 G(0, t)| < C,$$
 for some $C>0$.

 For large $n$, 
    $k_{n+1}/k_n \leq 1+\de$ and  
    $$\la_0 -\eps \leq \frac{1}{k_{n+1}}\log |\xi'_{k_{n+1}}(0)|. $$
    Taking $\eps, \de>0$ sufficiently small, we deduce that $$\frac{\la_0}{2} \leq \frac1m \log |\xi'_m(0)|$$ for 
    $m$ between $k_n$ and $k_{n+1}$, and thus for all large $m$.
    From this and bounded distortion \eqref{equBD}, there exists $M\geq 1$ for which, for $n\geq 1$ and $M\leq m\leq k_n$,   
    $$|\xi'_{m}(t)| \geq e^{\frac{m\la_0}{2} -1} > C$$
    for all $t \in \om_n.$
    Thus $0 \notin \xi_{m-1}(\om_n)$ for such $m,n$.
Now we deal with $m\leq M$. 
As $|\om_n| \to 0$ and the critical orbit of $g_0$ is not periodic,  there exists $N$ such that if $m\leq M$ and $n\geq N$, $0 \notin \xi_m(\om_n)$. 
\end{proof}

\noindent \textbf{Safe elements.}
Given a tent map $T$, we call an element $x$ of $T^{-N}(0)$ \emph{safe} if $x \notin \Orb(0)$. 
Considering the minimal $n$ with $T^n(x) = 0$,
safe elements admit continuations  under small perturbations of $T$.
Denote by  
\begin{equation}
\label{equDefSN}
\SS_N:=\bigcup_{j= 1 }^N T_{b_0}^{-j}(0) \setminus\Orb(0)
\end{equation}
   the collection of safe elements of $T_{b_0}$ for times up to $N$.
    Denote by $\SS = \SS_\infty$ the collection of all safe elements. This set is dense in $\hat I_{b_0}$. 
Let $\phi$ be the semi-conjugacy of $g$ to $T_{b_0}$
\begin{equation}
\label{equDefPhi}
	\phi \circ g = T_{b_0} \circ \phi.
\end{equation}
For each safe element $x$ with corresponding minimal $n$, there may be an interval mapped by $\phi$ to $x$. However, there is a unique point of $g^{-n}(0)$ mapped by $\phi$ to $x$. 
If $\phi$ is a conjugacy, then $\phi$ is bijective on $\phi^{-1}(\SS)$ and the points in $\phi^{-1}(\SS)$ admit continuations (as points of $g_t^{-n}(0)$ as one perturbs $g$).

    \begin{lem} \label{lemSafe3}
        Suppose $h$ is not locally constant at $t=0$. 
        Given $r_0>0$, there exist $\eps_0,N, N_1>0$ for which the following holds. 
        If $0<\eps<\eps_0$ and $n\geq N_1$ are such that $\xi_n$ is diffeomorphic on $[0,\eps]$ and $$|\xi_n([0,\eps])|\geq r_0.$$ 
          
        Then $$\# \left(\SS_N \cap \phi(\xi_n([0,\eps])) \right) \geq 3.$$
    \end{lem}
\begin{proof}
    If $\phi$ is a conjugacy, then $T$ is not periodic, all preimages of $0$ are safe and $\phi^{-1}(\SS)$ is dense in $I$. In this case, the lemma is trivial.

    Henceforth suppose $\phi$ is not a conjugacy, which implies $g$ is renormalisable of some type other than Feigenbaum. Since the entropy is not locally constant, there are $t$ arbitrarily close to $0$ for which $g_t$ is not renormalisable of the same type as $g$. Since $G$ is \wg, $g$ is not parabolic.
    It follows that the final renormalisation of $g$ is conjugate to the Chebyshev map and any previous renormalisation is of Feigenbaum type. 
    
    In particular, high iterates of $0$ periodically lie on the boundary of the smallest restrictive interval $(p,q)$, say at $p$. Set $p_0=p$, $q_0=q$, $p_k = g^k(p)=g^k(q)$ and $q_k = g^k(g^m(0))$, for $k = 1, \ldots, m-1$, where $m\geq 3$ is minimal with $g^m([p,q])\subset [p,q]$. 
    For each $k$, the restriction 
    $$g^m : [p_k, q_k] \to [p_k, q_k]$$ is unimodal. 
    Each interval $[p_k, q_k]$ is the preimage of a point under $\phi$. 
    Therefore the points $p_0, \ldots, p_{m-1}$ are all in the accumulation set of $\phi^{-1}(\SS)$. 
    Remark that 
    $$\phi^{-1}(\SS) \subset I \setminus \bigcup_{k=0}^{m-1}[p_k, q_k].$$
    Using continuations of points in $\phi^{-1}(\SS)$ for $\eps > 0$ arbitrarily small, we only need to show that the interval $\xi_n([0,\eps))$ is disjoint from $\cup_{k=0}^{m-1}(p_k, q_k)$. We may assume that $r_0$ is smaller than the distance between any two intervals $(p_k,q_k)$. Those intervals are not adjacent, as $g$ is renormalisable of some type other than Feigenbaum.

    The points $p_k$, $q_k$ have continuations  $p_k(t), q_k(t)$ with 
     $$g_t^m(p_k(t))  = g_t^m(q_k(t)) = p_k(t).$$
    In this preperiodic transversal setting, $|\xi'_n(0)| \to \infty$ as $n \to \infty$, as it is comparable to $|(g^n)'(v)|$, by equation \eqref{equQn}. 
    Hence
    $$|p_k'(0)|, |q_k'(0)| < |\xi_n'(0)|$$
    for all $k$ and all large $n$. 

    Let $n$ be large and suppose $\xi_n(0) = p_k$. Relative to $p_k(t)$ at $t=0$, $\xi_n(t)$ can go in two directions. 
    In one direction, $\xi_n(t)$ goes inside the interval $(p_k(t), q_k(t))$. If this happens, 
    $$g_t^m([p_k(t),q_k(t)]) \subset [p_k(t), q_k(t)]$$ 
    for all sufficiently small $t \geq 0$, and $g_t$ has a restrictive interval $(p_0(t),q_0(t))$.  This we exclude, for otherwise the entropy would be locally constant at $t=0$, a contradiction. 
    
    Consequently, $\xi_n(t)$ starts off going in the other direction and, for all sufficiently small positive $\eps'$, $\xi_n([0,\eps'))$ is an interval bordering but disjoint from $(p_k, q_k)$. 
    Thus, if $\xi_n$ is diffeomorphic on  $[0,\eps)$,
    $$\xi_n([0,\eps))\cap (p_k(0),q_k(0)) = \emptyset,$$
    which completes the proof.
\end{proof}

Let us recall that we work with a \wg unimodal family $G$ and that $\SS_N$ and $\phi$ are defined by \eqref{equDefSN} and \eqref{equDefPhi}.

\begin{prop} \label{propContent}
    Given $K_0>0$ and $N\geq 1$, there exist  $K \geq 1$ and $\eps_0, \gamma >0$ such that
    the following holds. 
        If $\om = [0,\eps)$, for some $\eps \in (0,\eps_0)$, $n \geq 1$ is large, $\Dist(\xi_n, \om)  \leq K_0$, $0 \notin \cup_{j=1}^{n-1} \xi_j(\om)$ and $\phi(\xi_n(\om))$ contains (at least) three safe elements of $\SS_N$, then there exists $t \in \om$ with 
    $$\gamma|\om| <  t < |\om|$$
    and 
    \begin{equation}\label{equContent}
        K^{-1} t^{\frac{h(0)}{\lambda_n}}
        \leq
        |h(t) - h(0)| 
        \leq
        K t^{\frac{h(0)}{\lambda_n}},
    \end{equation}
    where
    $$\lambda_n := \frac{1}{n} \log |\xi_n'(0)|.$$
    
    \end{prop}

    \begin{proof}
        Let $\theta>0$ denote half the minimal distance between any pair of elements of $\SS_N$ or of $\phi^{-1}(\SS_N)$. 
        Let $q$ be a middle safe element of $\SS_N$ in $\phi(\xi_n(\om))$ and $p \in g^{-j}(0) \cap \xi_n(\om)$ with $\phi(p)=q$, for some minimal $j \leq N$. 
        Then 
        $$
        \dist(q, \phi(\xi_n(\partial \om))), 
        \dist(p, \xi_n(\partial \om)) \geq 2\theta. 
        $$

        Let $p(t),q(t)$ denote the continuations of $p,q$ (it makes sense to write $q(t)$, since to each $g_t$ there corresponds a unique $T_{\exp(h(t))}$). 
        As $\SS_N$ is finite, we can choose $\eps_0 > |\om|$ such that $p(t)$ and $q(t)$ are always well-defined on $\om$ and 
        such that 
        $$
        \dist(q(\om), \phi(\xi_n(\partial \om))), 
        \dist(p(\om), \xi_n(\partial \om)) > \theta. 
        $$
    Hence
     there exists $t'$ with $\xi_n(t') = p(t')$. For this $t'$, 
    $$\theta <|\xi_n(0)-\xi_n(t')| < |\xi_n(\om)| \leq 4.$$ 
    By bounded distortion, 
    \begin{equation}\label{equGamK0}
    	\frac{\theta}{e^{K_0}|\xi_n'(0)|} < t' < \frac{4e^{K_0}}{|\xi_n'(0)|},
    \end{equation}
    and we can choose 
    $$\ga := \frac{\theta} {4 e^{K_0}}.$$

    Let $b' := \exp(h(t'))$. Now the corresponding semi-conjugacy maps $p(t')$ to 
    $$q(t')  = T_{b'}^n(1),$$
     and $b_0 \ne b'$ since  
    $$\theta < |q(t') - T_{b_0}^{n}(1)| < |\hat I_{b_0}|.$$
We introduced in \eqref{equDefXi} the notation $\Xi_k(b) = T_b^k(1)$, the equivalent of $\xi_k$ for tent maps.
    From equation \eqref{equKnead} and the hypothesis $0 \notin \cup_{j=1}^{n-1} \xi_j(\om)$, we deduce $\Xi_k(b) \neq 0$ for all $k=1,\ldots,n-1$ and all $b\in(b_0, b')$.    
     By \rlem{BruMisLem}, $\Xi_n$ is a diffeomorphism on $(b_0, b')$ and 
    $$
    |\Xi_n'(b)| \mc b_0^n$$
    for all $b \in (b_0, b')$. 
    Meanwhile $\theta < |\Xi_n(b_0) - \Xi_n(b')| \leq  |\hat I_{b_0}|$. Hence,
     $$ |b_0 - b'| \mc b_0^{-n}. $$
    As $|h(0) - h(t')| 
    = |\log b_0 - \log b'| \mc |b_0 - b'|$ 
    and 
    $$
    |\xi_n'(0)|^{\frac{h(0)}{\lambda_n}} =
    \exp(nh(0))=b_0^n,$$
    we deduce, using~\eqref{equGamK0} for the final step, that 
    $$
    |h(0) - h(t')| \mc b_0^{-n} 
    \mc |\xi_n'(0)|^{-\frac{h(0)}{\lambda_n}} 
    \mc
    t'^{\frac{h(0)}{\lambda_n}} ,$$
     noting $e^{K_0}$, $\theta$, $h(0)$ are constants and $\la_n$ converges to $\la_0 > 0$ as $n \to \infty$. 
    \end{proof}

    \begin{proof}[Proof of \rthm{HolderUnimodal}]
        Given a \wg unimodal family $G$,
        we must show that, if $h$ is monotone and not locally constant at $t=0$,  
        $$
        \lim_{\substack{t \to 0}} \frac{\log |h(t) - h(0)|}{\log t} = \frac{h(0)}{\lambda_0}.
        $$
        By monotonicity, it suffices to find sequences of parameters $t_n \searrow 0$ with 
        $$
        |h(t_n) -h(0)| \mc t_n^{\frac{h(0)}{\la_0}}
        $$
        and with 
        $$ \frac{\log t_{n+1}}{\log t_{n}}$$ 
            arbitrarily close to $1$. 

        Let $\delta, \eps > 0$ and let $r_0,m_0, k_n, \om_n$ be given by Lemma~\ref{factLS}. 
        From Lemma~\ref{factLS}, we have $|\xi_{k_n}(\om_n)| > r_0$,  
        $$1 <\frac{k_{n+1}}{k_n} \leq 1+\delta,$$
        $\Dist(\xi_{j}, \om_n) \leq 1$ for $j = m_0, \ldots, k_n$, and, for $t \in \om_n$,
        $$
        \la_0-\eps \leq \frac1{k_n} \log |\xi'_{k_n}(t)| \leq \la_0 +\eps.$$

        By \rlem{Safe3}, there is some fixed $N$ for which, for all large $n$, 
        $\phi(\xi_{k_n}(\om_n))$
        contains at least three safe elements of $\SS_N$. By \rlem{NCZ}, we can apply \rprop{Content} to obtain $K\geq1$, $\gamma >0$ and a sequence of parameters $t_{n}$ for which the following holds: 
            $$
            \gamma |\om_n| \leq t_{n} \leq |\om_n|
            $$
            and
            \begin{equation}\label{equLALA}
            K^{-1}t_{n}^ \frac{h(0)}{\la_0-\eps}
            \leq
            |h(t_n)-h(0)|
            \leq
            K t_n^\frac{h(0)}{\la_0+\eps}.
        \end{equation}

        It remains to show that the $t_n$ do not decrease too fast. 
            By bounded distortion (and  assuming $n$ is large),
            $$
            1-\eps \leq \frac {\log |\xi_{k_n}'(0)|} {-\log t_n}  \leq 1+\eps.
            $$
            We deduce
            \begin{equation}\label{equAiLALA}
                \frac{\log t_{n+1} }{\log t_{n}} \leq (1+\de)\frac{(\la_0+\eps)(1+\eps)}{(\la_0-\eps)(1 - \eps)}. 
        \end{equation}
        Since $\eps, \de>0$ are arbitrary, this completes the proof. 
        \end{proof}

        \subsection{Parabolic parameters} \label{sectPAR}
   \label{sectParab}

   \noindent\textbf{Quadratic maps.} We call a map $f_a$ parabolic if it has a periodic point $p$ (of period $k$, say) with multiplier $1$ or $-1$. 
If the multiplier is $-1$, then $f_a^k$ is orientation-reversing, locally. By Singer's theorem (using negative Schwarzian derivative), $p$ is attracting on at least one side, so it is attracting on both sides. Arbitrarily small neighbourhoods of $p$ are mapped diffeomorphically and compactly inside themselves by $f_a^{2k}$. As one perturbs $f_a$, this topological situation persists and there is an (at least one-sided) attracting periodic orbit of period $k$ or $2k$. Again by Singer, the critical point lies in the basin of attraction. The kneading itinerary of the critical point remains the same and the entropy is locally constant. 
If the multiplier is $1$, the entropy may vary with $a$ (as $a$ increases). In this subsection, we  prove that the entropy function is infinitely flat at parabolic parameters for general unimodal families.

\medskip

\noindent
    \textbf{Tent maps.} Multiplication by $b$ is the same as iteration of the function whose graph is given by $y = bx$. After $n$ iterates starting from $x$, one is at a distance $b^nx$ from $0$. Translating this behaviour, we understand iterates of a piecewise-linear map restricted to one branch. 
    
    Let $T_{b_0}$ be a preperiodic tent map and let $\alpha, k,N$ satisfy  
    $$\alpha := T_{b_0}^{N+k}(0)= T_{b_0}^N(0),$$ 
    with $k \geq 1, N\geq 0$ minimal with this property. Note $\alpha$ is periodic with period $k$. 
    Near $\alpha$, if $N = 0$ the graph of $T_{b_0}^k$ is on one side only of the diagonal; if $N >0$ the graph of $T_{b_0}^{k}$ crosses the diagonal at $\alpha.$ For nearby parameters $b$, if $N>0$ the graph still crosses the diagonal and, in a small neighbourhood of $\alpha$, there is only one branch. 
    If $N=0$, we need the following. 
    \begin{lem}
    Let $T_b$, $b \in (1,2]$ be a tent map. If $T_b^k(0)= 0$ for some $k \geq 1$ then $b^k \geq 2\sqrt{2}$. 
    \end{lem}
    \begin{proof}
    If $b_0 \in (2^{2^{-m-1}}, 2^{2^{-m}}]$, then $T_{b_0}$ is $m$-times renormalisable of period $2$.  It follows that $k = {2^{m}}p$ for some $p \geq 3$ and  $b^k \geq 2^{p/2}$. 
    \end{proof}
    \begin{lem} \label{lemUalpha}
    There exist $C_1, l\geq 1$ and 
    a neighbourhood $U_\alpha$ of $\alpha$ for which, for all $b$ close to $b_0$, the following holds. If $n\geq 1$ and  
    $$ T_{b}^{N+jk}(0) \in U_\alpha$$
    for $j = 1,\ldots,n$, then
    $$
    |\log b_0 - \log b| \leq C_1 b_0^{-nk/l}.$$
\end{lem}
    \begin{proof}
        There are three cases. First if $N>0$ and $U_\alpha$ is small, then iterates in $U_\alpha$ all lie in the same branch of $T_b^k$.  If $\alpha_b$ denotes the continuation of the periodic point $\alpha$, then 
        $$|T_b^{N+2k}(0)
        -T_b^N(0)| \geq  
        |T_b^{N+2k}(0)-
        \alpha_b|.$$ 
        Hence
     $$b^{(n-2)k}|T_b^{N+2k}(0) - T_b^{N}(0)| < |U_\alpha|.$$ 

     Next if $N =0$ and the graph of $T_b^k$ does not touch the diagonal in a small neighbourhood, then for $U_\alpha$ small enough, $T_b^{jk}(0)$ for $j=1,\ldots,n$ are monotone and  lie in the same branch.  We obtain
     $$b^{(n-1)k}|T_b^{k}(0) - 0| < |U_\alpha|.$$ 

     If $N=0$ and the graph crosses the diagonal, $\alpha$ splits into an orientation-preserving fixed point (for $T_b^k$) $\alpha_b$ and an orientation-reversing fixed point. Since the slope of $T_b^k$, for $b$ near $b_0$, is at least $2.75$, 
     the sequence $0, \alpha_b, T_b^{2k}(0)$ is monotone and 
     $$
     | T_b^{2k}(0) - \alpha_b| \geq |T_b^{k}(0) - 0|.$$
     Furthermore, iterates $T_b^{jk}(0)$, $j=2,\ldots, n$ are all in the same branch (monotonically receding from $0$).
     Thus 
     $$b^{(n-2)k}|T_b^{k}(0) - 0| < |U_\alpha|.$$ 

    Moving to  the next stage of the proof, $T_b$ is piecewise degree $1$ in $b$ and $x$ so
    $$T_b^{N+Mk}(0) - T_b^{N}(0),$$ 
    for $M \in \{1,2\}$, 
    is (piecewise) a polynomial of degree $N+Mk$ in $b$. Let $1 \leq l \leq N+Mk$ be the order of the zero $b_0$. The entropy changes with $b$, so the polynomial is not constant $0$ and $l$ is finite. 
    [One can presumably show $l=1$, but this is not needed.] 
    Then
    $$|T_b^{N+Mk}(0) - T_b^{N}(0)| \mc |b_0 - b|^l.$$
    Consequently $|b_0 - b|^l < C_0/b^{2nk}$, 
    whence
    $$
    |\log b_0 - \log b| \leq C_1 b_0^{-2nk/l}.$$
\end{proof}

\medskip
\noindent\textbf{$\cC^2$ unimodal families.} 
%
The following standard lemma will provide a lower bound for the time spent by the critical orbit in the funnel in the latter case. We use $\rho_t$ to study the local behaviour around a parabolic fixed point which is at least one-sided attracting.  
We may assume it is attracting to the right. 

\begin{lem}
\label{lemParab}
Let $\rho_t:(-2\de,2\de)\to \R$, $t \in [0,2\de^2)$ be a $\cC^2$ family of $\cC^2$ diffeomorphisms with $0 < \de < 1$. Let $\rho:=\rho_0$ and assume that $\rho(0)=0$, $\rho'(0)=1$ and that, for all $x \in (0,\de]$,
\begin{equation} \label{eq:rhox}
\rho(x)<x.
\end{equation}
    There exist arbitriraly small $ \eps_{0}, t_0>0$ and some $C >1$ such that  for all $t \in [0,t_0]$ and all $x \in (\sqrt{t}, \eps_0),$
    \begin{equation} \label{eq:xCt}
     \rho^n_t(x) \in (0, \eps_0)
     \mbox{ for all } n \leq \frac{1}{C\sqrt{t}}.
    \end{equation}
\end{lem}
\begin{proof}
    Take $\eps_0,t_0$  small enough. Then for $t \in [0,t_0]$
    $$\rho_t' > 3/4$$ on $[0, \eps_{0}]$. 
    Moreover, by continuity, the derivative estimate and~\eqref{eq:rhox},
\begin{equation} \label{eq:rhottenx}
    x/2 < \rho_t(x) < x
\end{equation}
    for $x \in [\eps_{0}/2, \eps_{0}]$, while, for some $C_0>0$ independent of $t$, 
\begin{equation} \label{eq:rhottenx2}
    x/2 - C_0t <  \rho_t(x) < \eps_{0}
\end{equation}
     for $x \in [0, \eps_{0}]$. 
     Consequently, if $x \in [0,\eps_0]$ and if $n$ is minimal such that $\rho_t^n(x) \notin [0, \eps_{0}]$, then $\rho_t^n(x) < 0$. 
     By monotonicity, for a given $x$, either $\rho_t^{j+1}(x) \geq \rho_t^j(x)$ for all $j\geq 0$, in which case $n = +\infty$ and the lemma holds, or
     $$\rho_t^{j+1}(x) < \rho_t^j(x)$$
     for $j < n$. 
     It remains to treat this second case.

     Given $x \geq \sqrt{t}$, there may be a first iterate $y$ of $x$ with $y \in [0, \sqrt{t})$. 
Then by~\eqref{eq:rhottenx2}, 
    $$\sqrt{t}/2 - C_0t  < y < \sqrt{t}. $$
    For $t$ small, $y \geq \sqrt{t}/3.$ 
         We shall estimate the time needed for this iterate $y$ to leave. 

By hypothesis and Taylor expansion at $(0,0)$, there exist $C_1>0$ and $\al \in \R$ such that for all $t \in [0,\de^2)$, $x \in (-\de,\de)$,
\begin{equation}
\label{equParab}
|\rho_t(x)-x-\al t| \leq C_1 (|x|+|t|)^2.
\end{equation}
For $0 \leq x \leq \sqrt{t}$, setting $C = |\al| + 4C_1$,
$$x - \rho_t(x) \leq Ct.$$
    Hence, 
    $$
    y - \rho_t^j(y) \leq jCt \leq y$$
    for $j \leq 1/3C\sqrt{t} \leq y/Ct.$ In particular, for this range of $j$, $\rho_t^j(y) \geq 0$. 
    This completes the proof. 
\end{proof}

Remark that in the above proof, $\rho_t$ may still have a fixed point for some $t>0$. By~\eqref{eq:rhottenx}, the rightmost one in $[0,\eps_{0}]$ must be attracting on its right. For $t$ small, it must also be close to $0$. 

We are now in position to prove the following result, which has \rthm{Parab} as an immediate corollary.
\begin{prop}
\label{propParab}
Let $G$ be a $\cC^2$ family of $\cC^2$ unimodal maps, $G(\cdot, t) = g_t(\cdot)$, so that $g_0$ has a parabolic periodic orbit $p$ which attracts, but is disjoint from, the critical orbit. Then for $t>0$ sufficiently small
$$\frac{\ln |h(t) - h(0)|}{\ln t} \gtrsim \frac{-1}{\sqrt t \ln t}.$$
\end{prop}
\begin{proof}
We may suppose that the entropy is not locally constant at $t=0$, with the convention $\log 0 = -\infty$. Furthermore, we suppose that the critical point of each $g_t$ is $0$. 
The kneading sequence $K(0)$ of $g_0$ is preperiodic, so 
$$K(0)\ne K(f_{a_F}),$$
the kneading sequence of the Feigenbaum map. Therefore if 
$h(0)=0$ then $h=0$ on some neighbourhood $[0,t_0)$. Henceforth assume $h(0)>0$. 

Let $b_0 := \exp(h(0)) > 1$, so $g_0$ is semi-conjugate to the (necessarily preperiodic) tent map $T_{b_0}$. Let $N\geq 0, k\geq 1$ be minimal with 
    $$ \alpha := T_{b_0}^{N+k}(0)=T_{b_0}^N(0).$$
    Let $\phi_t$ denote the semi-conjugacy between $g_t$ and $T_{\exp(h(t))}$. 
    Set $V = \phi_0^{-1}(\alpha)$ and let $p\in V$ denote the parabolic periodic point whose orbit attracts $0$. 
    Note that $V$ contains the immediate basin of attraction of $p$ and  
    $g_0^{N+jk}(0)$ is in the interior of $V$ for all $j\geq 0$.
     Let $U_\alpha$ be given by~\rlem{Ualpha}. By continuity of preperiodic points, for example, there is a fixed neighbourhood $U$ of $V$ with $U \subset \phi^{-1}_t(U_\alpha)$ for all small $t$. 
     For some $M\geq 0$, 
     $g_0^{N + Mk}(0)$ is sufficiently close to $p$ that one can apply~\rlem{Parab}. 
     We obtain, for small enough $t$, 
    \begin{equation}\label{eqn:parV2}
        g_t^{N+ k}(0), g_t^{N+2k}(0), \ldots, g_t^{N+ k\lfloor Ct^{-1/2} \rfloor}(0) \in \phi_t^{-1}(U_\alpha). 
    \end{equation}
        By \rlem{Ualpha}, 
    for $b = \exp(h(t))$,
    $$
    |\log b - \log b_0| \leq C_1 b_0^{k\lfloor Ct^{-1/2}\rfloor /l}.
    $$
    Taking log, 
    $$
    \log|h(t) - h(0)| \leq C_2t^{-1/2} + \log C_1.$$
\end{proof}

%% file: measure.tex
In this section we collect some facts concerning large-measure sets of parameters in the quadratic family $f_a : x \mapsto x^2 + a$, $a\in \ip$. 

A parameter $a$ is called Collet-Eckmann if $\lla(a) > 0$. 
Let 
$$E := \{a \in [-2,a_F]\ :\ \lla(a)>0 \text{ and }\liminf_{n \to \infty} 
	\frac 1n \ln |f_a'(f_a^n(a))|=0\},$$
be the set of Collet-Eckmann parameters whose critical obit is \emph{slowly recurrent} (or non-recurrent).

Set $\xi_n(a):=f_a^n(a)$. Computing gives
\begin{equation}
\label{equXiDer}
\xi_k'(a)=1 + f_a'(f_a^{k-1}(a)) + \ldots + (f_a^{k-1})'(f_a(a)) + (f_a^{k})'(a).
\end{equation}
Consider (compare~\eqref{equQn})
\begin{equation}
\label{equQ}
	Q_n(a):=\frac{\xi_{n}'(a)}{(f_a^n)'(a)}=\sum_{k=0}^n \frac 1{(f_a^k)'(a)}
\end{equation}
 If the limit of $Q_n(a)$ as $n \to \infty$ exists and is non-zero, we say that $f_a$ satisfies the \emph{transversality condition}. 
 Recall that Levin \cite{Lev} proved transversality holds if
$$\sum_{k \geq 0} \frac 1{|(f_a^k)'(a)|} < \infty.$$
For Collet-Eckmann parameters the sum is finite, so transversality does hold
at each $a \in E$. Nowicki \cite{Now} proved that the Collet-Eckmann condition implies uniform exponential growth at preimages of the critical point. By Singer's theorem \cite{Sin}, since  $\lla(a) >0$, all periodic orbits are hyperbolic repelling.

Therefore, all the hypotheses required to apply 
     \cite[Theorem~1]{Tsu} for parameters in $E$ are verified. 
We obtain:

\begin{fact}[{\cite[Theorem~1(I)]{Tsu}}] \label{factEwr}
    Almost every $a \in E$ satisfies Tsujii's weak regularity condition~\eqref{equWR}.
\end{fact}

\begin{prop}[{\cite{Tsu, AviMorIHES}}] \label{propQuadwr}
    The set
    $$
\JWR := \{a \in \H^c : \la(a) >0 \mbox{ and } f_a \mbox{ verifies } \eqref{equWR} \}.
$$
has full measure in $\H^c$ (and thus in $\F^c$). Moreover, for almost every $a \in \JWR$, 
$$\la(a) = \chi(\acip^a).$$
\end{prop}
\begin{proof}
    Avila and Moreira \cite{AviMorIHES, AviMor} (building on Lyubich's \cite{Lyu2002Ann}) showed that $E$ has full measure in $\H^c$, that $\la(a)$ exists and that $\la(a) = \chi(\acip^a)>0$ for almost every $a \in E$. 
    Fact~\ref{factEwr} implies~\eqref{equWR} holds almost everywhere in $E$. 
\end{proof}

    Let $Y$ be given by~\eqref{equYdef}. 
    Recall that  $a \in Y$ if 
    $f_a$ satisfies~\eqref{equWR} and  
    $$\lambda(a) = \chi(\mme^a) >0.$$

\begin{prop} [{\cite{Bru, San}}] \label{prop:tentwr}
    The set $h(Y)$ has full measure in $[0, \log 2]$.
\end{prop}
\begin{proof} 
    For almost every value of the entropy, the weak regularity condition was shown in the doctoral thesis of Sands~\cite{San}. This can be extracted as follows: his Theorem~54 states that, for almost every value of the entropy, the kneading invariant of the corresponding tent map is ``slowly recurrent''\footnote{Sands' definition of slow recurrence differs from the standard one (which we use) and more resembles Tsujii's weak regularity condition}. Slow recurrence is defined on page~57 in terms of a function $\Cal R$, itself defined in section~2.3. $\Cal R(j)$ is large if $|f^j(0)|$ is small. Combined with Lemma~39, one obtains weak regularity [see also his footnote on page~57] for the quadratic map.  

    For almost every $w \in [0, \log 2]$, there is a unique $a_w \in \ip$ with $h(a_w) = w$. 
    By 
\cite[Corollary~1]{Bru},
for almost every $w \in [0, \log 2]$, for the map $f_{a_w}$, the critical point $0$ is typical with respect to the measure of maximal entropy. 
With Birkhoff's theorem, we nearly obtain $\la(a_w) = \chi(\mme^{a_w})$; the only missing ingredient (since $\log|f_{a_w}'|$ is unbounded at $0$) is the weak regularity condition, which we have shown.
\end{proof}

We say a parameter $a$ is \emph{Misiurewicz} if all periodic orbits of $f_a$ are hyperbolic repelling and $0$ is non-recurrent. If $a$ is  non-renormalisable, the entropy is at least $\frac{\log2}2$ and is not locally constant at $a$. 
    \begin{fact}[{\cite[Theorem~1.30]{DobTod}}] \label{fact:DT30}
        Given any $\eps >0$ and any non-renormalisable Misiurewicz parameter $a_0$, there is a non-renormalisable Misiurewicz parameter $a$ arbitrarily close to $a_0$ such that $\acip^a$ exists and $h(\acip^a) \in (0,\eps)$. 
    \end{fact}

    \begin{prop} \label{prop:smallh}
        Given $\eps >0$, 
        there exists a positive measure set $A$ of non-renormalisable Collet-Eckmann parameters with $0 < \la(a) = \chi(\acip^a) < \eps$ for all $a \in A$. 
    \end{prop}
    \begin{proof}
        By Fact~\ref{fact:DT30}, there exists a non-renormalisable Misiurewicz parameter $a_1$ with $h(a_1) > \frac{\log 2}2$ and $h(\acip^{a_1}) \in (0,\eps/2)$. 
        By \cite[Main Theorem]{Tsu95}, there is a positive measure set $A'$ of parameters, with $a_1$ as a density point, such that, for any sequence $a_p \to a_1$, $\acip^{a_p}$ exists (and is necessarily unique) and converges to $\acip^{a_1}$, while by \cite[Inequality~(2.1)]{Tsu95},
    $$
    \limsup_{a_p\to {a_1}} h(\acip^{a_p}) \leq h(\acip^{a_1}).$$
    Entropy equals Lyapunov exponent for acips. For $r$ small enough and $A_r = A'\cap B(a_1,r)$, we deduce 
    $$
    0 < \chi(\acip^{a}) < \eps$$
    for all $a \in  A_r$. By \rprop{Quadwr}, for almost every $a \in A_r$, $\la(a) = \chi(\acip^a)$.

    From Tsujii's (and Benedicks-Carleson's) construction, for parameters in $A'$, there are arbitrarily small neighbourhoods of $0$ mapped to some fixed large scale. On the other hand, $f_{a_1}$ is non-renormalisable, so for nearby renormalisable parameters, the restrictive intervals are very small and do not get mapped to the fixed large scale. Therefore, they are not in $A'$. Hence, taking $r$ small enough, all parameters in the set $A_r$ are non-renormalisable. Set $A = A_r$ to complete the proof. 
\end{proof}

%% file: typical.tex
For large sets of parameters, one can compare the Lyapunov exponent of an invariant measure with the Lyapunov exponent along the critical orbit. This allows one to estimate $\frac{h(a)}{\la(a)}$ on large sets and constitutes the goal of this section.

\subsection{Visible measures of maximal entropy} \label{S:Vmme}
The most interesting dynamical measures for a unimodal map are the measure of maximal entropy $\mme$ and the absolutely continuous invariant probability measure $\acip$ (supposing the latter exists). 
To prove Theorem~\ref{thm:Zdun}, it remains to show that if $\mme^a = \acip^a$ for a quadratic map $f_a : x \mapsto x^2 +a$, then $a = -2$. 
Recall \rdef{PreC} of pre-Chebyshev maps. 
\begin{prop}
    The only pre-Chebyshev quadratic map $f_a$ is $$f_{-2} : x \mapsto x^2-2.$$ 
\end{prop}
\begin{proof}
We shall need  a \emph{quadratic-like} extension for the renormalised map. A quadratic-like map is a holomorphic proper map $g : U \to V \subset \C$ between topological discs with $\overline{U}$ compactly contained in $V$ and $g$ of degree $2$. 

By \cite[Theorem~5.20, p204]{JiangBook},
    if $f$ is a renormalisable quadratic map with restrictive interval $J$ and renormalisation $f^n:J \to J$, then $f$ viewed as a complex map $f : \C \to \C$ is renormalisable. The renormalisation is a quadratic-like map $f^n : U \to V$  with $U$ a complex neighbourhood of $J$. 

For a pre-Chebyshev quadratic map $f$, the final renormalisation $f^n : J \to J$ is smoothly conjugate on the interior of $J$ to $x \mapsto 1-2|x|$. The corresponding quadratic-like map has Julia set equal to $J$, since $J$ is backward-invariant. Harmonic measure on $J \subset U$ is equivalent to Lebesgue measure, itself equivalent to the measure of maximal entropy. 

There is extra rigidity in analytic setting which comes from a dichotomy of 
Zdunik and Popovici and Volberg concerning harmonic measure \cite{Zdu, PopVol}: for a polynomial-like map, either the measure of maximal entropy and the harmonic measure are mutually singular or the map is conformally equivalent to a polynomial (on some neighbourhoods of the maps' Julia sets). In particular, $f^n : U \to V$ is conformally equivalent on $U$ to $z \mapsto z^2-2$ on a corresponding neighbourhood of $[-2,2]$. 

With Inou's \cite[Theorem~1]{Inou}, we go a step further. Since $f$ and $g : z \mapsto z^2 -2$ are entire and have bounded degree and $f^n : U\to V$ and $g$ are conformally equivalent, $f^n$ and $g$ have the same degree on $\C$. Therefore $n=1$ and  $f = g$.  
\end{proof}

\noindent\textbf{Proof of Theorem~\ref{thm:Zdun}.} Combining the above proposition with the following fact completes the proof of Theorem~\ref{thm:Zdun}. 
 \begin{fact}[{\cite[Theorem~2]{DobMME}}] \label{factDobMME}
Given an S-unimodal map $g$ with positive entropy, the measure of maximal entropy is absolutely continuous if and only if $g$ is pre-Chebyshev.
\end{fact}

    \subsection{Comparing topological entropy and Lyapunov exponents}
    In this section, we consider the quadratic family $f_a$, $a \in \ip$, and the entropy function $h : a \mapsto \htop(f_a)$. 

Given a piecewise-smooth map $g$, we define 
    $\M(g) $ as the set of ergodic, $g$-invariant, probability measures with non-negative Lyapunov exponent.
    We have a pressure function 
    $$ P_g(t) := \sup_{\mu \in \M(g)} \{h(\mu) - t\chi(\mu)\}.$$
    For $g=f_a$, we write $P_a$ for $P_{f_a}$.
    As a supremum of lines (of slope $-\chi(\mu)$), the pressure is Lipschitz, convex and decreasing. A measure realising the supremum for $t$ is an \emph{equilibrium measure} for the parameter $t$.

    \begin{fact}\label{facts1}
        Let $g$ be a $\cC^2$ unimodal map with positive entropy and suppose $\mu \in \M(g)$.
        \begin{enumerate}
        \item
            \label{enum:facts11}
            $h(\mu) \leq \chi(\mu)$; 
        \item
            \label{enum:facts12}
            if $\chi(\mu) >0$ then
            $$
            h(\mu) = \chi(\mu) \iff \mu \mbox{ is an acip};$$
        \item
            \label{enum:facts14}
            if $\chi(\mu) >0$ then
            $\hd(\mu) = \frac{h(\mu)}{\chi(\mu)}$;
        \item
            \label{enum:facts141}
        $P_g(1) \leq 0$. 
        \end{enumerate}
    \end{fact}
    \begin{proof}
        \ref{enum:facts11}.\ is Ruelle's Inequality \cite{Rue} (note that $\chi(\mu) \geq 0$ by definition of $\M(f)$); \ref{enum:facts12}.\ is Pesin's formula \cite{Led}; \ref{enum:facts14}.\ is the Dynamical Volume Lemma \cite{HofRai, DobCusp}; 
        \ref{enum:facts141}.\ follows from the definition of pressure and~\ref{enum:facts11}. 
    \end{proof}
    As a consequence of Theorem~\ref{thm:Zdun} and Fact~\ref{facts1} we obtain the following inequalities.
    \begin{cor}
            \label{cor:facts123}
            If $a \in (2,a_F)$ and $\acip^a$ exists, then 
            $$\chi(\acip^a) = h(\acip^a) < h(a) = h(\mme^a) < \chi(\mme^a).$$
        \end{cor}
        In~\rsec{Intro} we set
    \begin{eqnarray*}
        X :=& \{a\in \V \cap \JWR  : \la(a) = \chi(\acip^a)\},\\ 
        Y :=& \{a\in \V \cap \JWR  : \la(a) = \chi(\mme^a)\} 
    \end{eqnarray*}
    and noted that  $-2 \notin X \cup Y.$
    By the previous corollary, for $a \in X$, $h(a) > \la(a)$, while for $a \in Y$, $h(a) < \la(a)$. 
In the sequel, we gather the necessary tools to obtain uniformity in these inequalities. Convergence of maps will be with respect to the sup-norm. 
\begin{prop} \label{prop:presconv}
    Let $(f_k)_{k\geq 1}$ be a convergent sequence  of maps with limit $f_0$, and suppose each $f_k$, $k \geq 0$, is an S-unimodal map with positive entropy and with derivative bounded by some $M>0$. 
    For each $t > -\htop(f_0)/{2 \log M}$ such that $\inf_{k \geq 0}P_{f_k}(t) > 0$ 
    	$$P_{f_0}(t) = \lim_{k\to\infty} P_{f_k}(t).$$
\end{prop}
\begin{proof}
    To apply a result in \cite{DobTod}, we will need a uniform lower bound on the entropy of the equilibrium states for $t, f_k$. Let $f := f_k$ for some $k\geq 0$. 

    For $t \leq 0$, $P_f(t) \geq \htop(f)$. For any measure $\mu$, $\chi(\mu) \leq \log M.$ Simple geometry implies that
    for $t \in \left[\frac{-\htop(f)}{2\log M}, 0\right]$, any equilibrium measure $\mu^f_t$ for $t, f$ must have entropy $h(\mu^f_t) \geq \htop(f)/2$. 
    Such equilibrium measures exist by upper-semicontinuity of metric entropy and Lyapunov exponents (for a fixed map).
    For smooth unimodal maps, $\htop$ depends continuously on the map \cite{MisJumps}.
    Thus for sufficiently  large $k$, 
    $$
    h(\mu^{f_k}_t) \geq \htop(f_0)/3.$$

    For $t >0$, $P_f(t) \geq \eps >0$ implies $h(\mu_t^f) \geq \eps$, if the equilibrium measure exists. It exists by \cite[Corollary~1.20]{DobTod}. 

    According to \cite[Remark~1.14]{DobTod}, for smooth unimodal maps,  the \emph{decreasing critical relations} hypothesis is unnecessary. 
    Applying \cite[Lemma~13.1]{DobTod}, the pressure converges as required. 
\end{proof}

\begin{prop}
    Given an S-unimodal map $f$ with positive entropy, the map $t \mapsto P_f(t)$ is real-analytic in a neighbourhood of $0$. 
    \end{prop}
    \begin{proof}
        We shall use the decomposition of Jonker and Rand 
        \cite[Theorem~1]{JonRan} (for unimodal maps) of the non-wandering set $\Omega(f) = \Omega_0\cup \Omega_1 \cup\cdots$ into strata. The topological entropy is carried by $\Omega_1$: if $h_i$ denotes the topological entropy of $f$ restricted to $\Omega_i$, then $h_0 = 0$, 
        $$\htop(f) = h_1 > h_j \geq h_{j+1}$$
        for $j\geq 2$. 
        For tent maps, the decomposition stops at $\Omega_1$. If it stops at $\Omega_1$ for $f$ too, let $W := \emptyset$. If it does not stop at $\Omega_1$ for $f$, then there is a maximal open interval $W\ni 0$ collapsed by the semi-conjugacy to the corresponding tent map, and $\Omega_j$ is contained in the (periodic) orbit of $W$ for $j \geq 2$. 
        Therefore, if we denote by $g$ the restriction of  $f$ restricted to the complement of $W$, the non-wandering set of $g$ is just $\Omega_0 \cup \Omega_1$. Now $\Omega_0$ is just the orientation-preserving fixed point(s), while $\Omega_1$ consists of a finite (possibly $0$) number of isolated periodic orbits and a transitive set $\Omega_1'$. 
        From the entropy estimates $\htop(f) > \sup_{j\geq2} h_j$, it follows that the pressures of $g$ and $f$ coincide on a neighbourhood of zero.

        Apply \cite[Theorem~1.28]{DobTod} for $g$, noting that a transitive point in $\Omega_1'$ is a transitive point in $J(g)$ \cite[Definition~1.25]{DobTod}, to obtain analyticity of the pressure function for $g$ on a neighbourhood of zero. An alternative is, provided $g$ has no parabolic points, to apply \cite[Theorem~A]{PrzRL} for $g$ and $\Omega_1'$.
    \end{proof}
    
\begin{thmbody} \label{thm:mmecns}
    Let $s\mapsto g_s$ be a continuous one-parameter family of S-unimodal maps with positive entropy. Then
    $$s \mapsto \htop(g_s), \quad s \mapsto \mme^{g_s}, \quad s \mapsto \chi(\mme^{g_s})$$
    are continuous.
\end{thmbody}
\begin{proof}
    The first claim was shown by Misiurewicz \cite{MisJumps}.
    The second claim was shown by Raith \cite{Rai}.
    Let us show the third. 
    The derivative of $P_{g_s}$ at $0$ exists, since $P_{g_s}$ is real-analytic on a neighbourhood of $0$. The value of the derivative is $-\chi(\mme^{g_s})$. Since the pressures converge (Proposition~\ref{prop:presconv}) on a neighbourhood of zero and the pressure functions are convex,  the derivatives converge. Consequently, the Lyapunov exponents converge. 
\end{proof}

\begin{lem} \label{lem:acipconv}
    Let $(g_k)_{k\geq 1}$ be a sequence  of maps converging to a map $g_0$, and suppose each $g_k$, $k \geq 0$, is an S-unimodal map with positive entropy and derivative bounded by some $M>0$. 
    Suppose each $g_k$, $k \geq 1$, has an acip $\acip^k$ and 
    $$\htop(g_k)/\chi(\acip^k) \to 1$$ as $k \to \infty$. 
    Then the measure of maximal entropy for $g_0$ is absolutely continuous.
\end{lem}
\begin{proof}
    We have $h(\acip^{k}) = \chi(\acip^{k})$ for all $k \geq 0$. 
    Since $\htop(g_0) > 0$ and $\htop$ is continuous, we deduce 
    $$
    \lim_{k\to\infty} h(\acip^{k})= 
    \lim_{k\to\infty} \chi(\acip^{k})= 
    \htop(g_0).$$ 
    Consequently, 
    $$
    \liminf_{k \to \infty} P_{g_k}(t) \geq \htop(g_0) -t \htop(g_0).$$
    Thus, by Proposition~\ref{prop:presconv}, $P_{g_0}(t) \geq \htop(g_0) - t \htop(g_0)$ for $t \in [0,1)$. Meanwhile, $P_{g_0}(1) \leq 0$, so $P_{g_0}(1) = 0$. 
        Since $P_{g_0}$ is convex and its graph passes through $(0, \htop(g_0))$ and $(1,0)$, its graph over the interval $[0,1]$ is in fact the straight line joining these points with slope $-\htop(g_0)$. 
        The graph of the line $\htop(g_0)-t\chi(\mme^{g_0})$ is tangent to $P_{g_0}$ at $t = 0$. 
        Since $P_{g_0}$ is  analytic on a neighbourhood of $0$, we deduce that $\chi(\mme^{g_0}) = \htop(g_0) = h(\mme^{g_0})$, which implies $\mme^{g_0}$ is absolutely continuous.
\end{proof}

    \begin{prop} \label{prop:epceps}
    Given $\eps>0$, there exists $\delta>0$ for which 
    \begin{itemize}
        \item
            for all $a \in  [-2+\eps, a_F-\eps]$, if $\acip^a$ exists then $h(a) / \chi(\acip^a) > 1+ \delta$;
        \item
            for all $a \in [-2+\eps, a_F-\eps]$,  $\chi(\mme^a)/h(a) > 1+\delta$.
    \end{itemize}
\end{prop}
\begin{proof}
   The first statement is an immediate corollary of Lemma~\ref{lem:acipconv} and Theorem~\ref{thm:Zdun}.
    The second follows from Theorem~\ref{thm:mmecns}. 
    \end{proof}

    \begin{prop} \label{propEpsdim}
        For all $\eps >0$ and $\de >0$ 
        $$
        \mathrm{Leb}(\{ a \in X \cap (-2, -2+\eps) : 1< h(a)/\la(a) < 1 + \de\}) > 0,$$
        $$
        \mathrm{Leb}(h\left(\{ a \in Y \cap (-2, -2+\eps) :1> h(a)/\la(a) > 1- \de\}\right)) > 0.$$
    \end{prop}
    \begin{proof}
        Freitas \cite{FreSRB} showed that 
         $-2$ is a one-sided Lebesgue density point of a positive measure set of (Benedicks-Carleson) parameters on which $a \mapsto h(\acip^a)$ and thus $a \mapsto \chi(\acip^a)$ vary continuously, which implies the first statement. 

        The second statement follows from Theorem~\ref{thm:mmecns} and Proposition~\ref{prop:tentwr}, noting that $h$ is not locally constant at $a =-2$. 
    \end{proof}

    \subsection{Dimension estimates}

    The following lemma is a variant of Lemma~5.1.3 from the course notes of Bishop and Peres, \emph{Fractal sets in Probability and Analysis}.
\begin{lem}
\label{lemFD}
Let $u:E \to \R$ be a real map defined on a set $E \se \R$. For every $\al > 0$ let the $\al$-flat set of $u$ be
$$F(u,\al):=\left\{ x \in E \ : \ \liminf_{\matop{y \to x}{y \in E}} \frac{\ln |u(y)-u(x)|}{\ln |y-x|} \geq \al \right\}.$$
If $A \se F(u,\al)$ then 
$$\hd(u(A)) \leq \frac{\hd(A)}\al.$$
\end{lem}
\begin{proof}
    Let $d > \hd(A)$, let $\eps >0$ and let $0< \beta < \alpha$. 
    Modulo a countable partition of $A$, we can assume that there is a $C >1$ for which, 
    for all $x \in A, y \in E$ with $|x-y| < C^{-1}$,
$$|u(y)-u(x)| \leq C|y-x|^{\beta}.$$
Let $\{U_j\}$ be a covering of $A$ of diameter $< C^{-1}$ such that $\sum_j |U_j|^d \leq \eps$. 
The covering $\{u(U_j)\}$ of $u(A)$ 
satisfies 
$$
\sum_j |u(U_j)|^{d/\beta} \leq C \sum_j |U_j|^{d} \leq C\eps.
$$
Since $d > \hd(A), \eps >0$ and $\beta < \alpha$ are arbitrary, this completes the proof.
\end{proof}

As $X,Y \subset \V$, $h$ is bijective on $X \cup Y$. In the remaining lines of this section, $h$ will stand for its restriction to $X \cup Y$. Applying \rthm{Holder}, for each $a \in X \cup Y \subset \V \cap \JWR$,
$$a \in F(h, h(a) / \la(a)) \text{ and } h(a) \in F(h^{-1}, \la(a) / h(a)).$$

\begin{proof}[Proof of Theorem~\ref{thm:main}]
    \rlem{FD} and Proposition~\ref{prop:epceps} now imply 
        $$\hd\left(h\left(X \cap [-2+ \eps, a_F - \eps] \right)\right), 
    \hd\left(Y \cap [-2+ \eps, a_F - \eps] \right) < 1$$
 for each $\eps >0$, as required.
\end{proof}

\begin{proof}[Proof of \rthm{SuperHolder}]
    Given $\eps'>0$, let  $A$ be given by Proposition~\ref{prop:smallh}. For $a \in A$, 
    $$
    \chi(\acip^a) = \la(a) \leq \eps',
    $$
    while $h(a) \geq \frac{\log 2}{2}.$
    By \rthm{Holder} and \rlem{FD}, $$\hd({h(A)}) \leq \eps' \frac{2}{\log 2} \hd({A}).$$
    Noting that $\hd({A}) =1$ and that $\eps'$ can be taken arbitrarily small, there exist positive measure sets $A$ with $\hd({h(A)}) < \eps$, as required. 
\end{proof}

\begin{thmbody} \label{thmHDim1}
        For every $\eps >0$, 
        $$
        \hd\left(h(X \cap (-2, -2+\eps) )\right) = \hd(Y \cap (-2, -2+\eps)) = 1.$$
    \end{thmbody}
    \begin{proof}
        This follows from \rprop{Epsdim} and \rthm{Holder}.
    \end{proof}

%% file: uniHol.tex
    
In this section we shall prove \rthm{UniHol}, namely that the entropy function is uniformly H\"older continuous for the quadratic family. 

Recall that, if $h(a) < \frac{\ln 2}{2^m}$, then $f_{a}$ is $m$ times Feigenbaum renormalisable and the $m$th renormalisation is  a unimodal map $g$ with entropy $$\htop(g) = 2^m h(a).$$

By the theory of renormalisation \cite[Theorem~1]{Sul}, there is a universal bound (independent of $m \geq 0$) $\Ga \geq 4$:
\begin{equation}\label{equSullbound} ||g'||_{\infty} \leq \Ga.\end{equation}
    Numerical estimates (not presented here) indicate that if $m \geq 1$, one can take $\Ga < 3$. 

Let $a_m$ denote the quadratic parameter with $h(a_m) = \frac{\log 2}{2^m}$. Note that  $a_0 = -2$, corresponding to the Chebyshev map. For each $m\geq0$, $f_{a_m}$ is $m$ times Feigenbaum renormalisable and the $m$th renormalisation is conjugate (on the restrictive interval) to Chebyshev. Again from renormalisation theory, $$ \lim_{m\to\infty} \left|\frac{a_{m}-a_{m-1}}{a_{m+1}-a_{m}}\right|
    =
   \de_* \approx 4.67.$$
   It follows that, given $\alpha_0 < \frac{\log 2}{\log \de_*}$, for some $C_0>1$, $h$ restricted to the set $\{a_m : m \geq 0\}$ is $(C_0, \alpha_0)$-H\"older. 
    To prove 
    \rthm{UniHol}, it therefore suffices to prove the existence of $C_1, \alpha_1$ for which, on each interval $[a_m, a_{m+1}]$, $h$ is $(C_1, \alpha_1)$-H\"older. 
    The following lemma is primarily due to Guckenheimer, but the proof uses Brucks and Misiurewicz' (Benedicks-Carleson-type) estimates for tent maps instead of studying kneading determinants. 

    \begin{lem}[{\cite[Lemma~3]{Guc}}]
        There is a constant $C>0$ such that, for every $n$ and $a,a' \in [a_m, a_{m+1}]$, if
        \begin{equation}\label{eq:guc1}
            |h(a) - h(a')| > C 2^{-m}2^{-n/2},
        \end{equation}
        then there is a periodic tent map $T_{\hat b}$ with period at most $2^m n$ for which $h^{-1}(\log \hat b) \subset [a, a']$. 
    \end{lem}
    \begin{proof}
        Recall that by definition \eqref{equDefXi}, $\Xi_n(b) = T^n_b(1)$, the $n$th iterate of the critical value for the tent map $T_b$. By~\cite{BruMis}, for some constant $\rho >0$ and all $b \geq \sqrt{2}$, 
        $$|\Xi'_n(b)| \geq \rho b^n$$
        wherever $\Xi_n$ is differentiable, and $\Xi_n$ is differentiable at $b$ unless $\Xi_j(b) = 0$ for some $j \leq n-1$. If $\Xi_j(b) = 0$, then $0$ is periodic of period $j+1$. 
        Let $b \geq \sqrt{2}.$
        Now 
        $$\Xi_n(b) \in [-1/(\sqrt{2} - 1), 1/(\sqrt{2} -1)],$$
        so if we look at the maximal parameter interval $(b,b')$ ($b' > b$) on which $\Xi_n$ is differentiable, 
        $$
        |b'-b| < \frac{2}{\sqrt{2}-1} b^{-n} \rho^{-1}.
        $$
        Noting $b,b' \geq \sqrt{2}$,
         there is a constant $C$ for which 
        $$
        |\log b'- \log b| < C 2^{-n/2}.
        $$
        Consequently, if one has $\sqrt{2} \leq b < b'$ and $\log b' - \log b > C2^{-n/2}$, there is a periodic tent map $T_{b_*}$ with period  $j \leq n$ and slope  $b_*$ lying strictly between $b$ and $b'$. 
         Then $T_{b_*^{2^{-m}}}$ has entropy $2^{-m}{\log b_*}$ and is periodic with period $2^mj$. 
        One immediately obtains the corresponding statement for renormalisable tent maps: If one has 
        $$2^{-m-1}\log 2 \leq \log b <\log  b' \leq 2^{-m}\log 2$$ and $\log b' - \log b > C2^{-m}2^{-n/2}$, there is a periodic tent map with period at most $2^m n$ and slope   in $(b,b')$. 

        Now with $m, a, a',n$ satisfying~\eqref{eq:guc1}, let $b, b'$ be the corresponding tent maps and $\hat b \in (b,b')$ a periodic tent map parameter with period at most $2^m n$. 
        Then $h^{-1}(\log \hat b) \subset [a,a']$, by monotonicity of entropy, completing the proof. 
    \end{proof}

 We will use the following simple but fruitful observation by Przytycki. If some critical value comes back too soon too close to the critical point $0$, then the map has an attracting cycle: 
\begin{lem}[{\cite{Prz}}] \label{lem:przyt}
    Let $C, \gamma\geq 1$ and let $g$ be a $C^1$ map with derivative satisfying  
    $|g'(x)| \leq \min(C|x|,\gamma)$.
    If $g$ does not have an attracting periodic orbit of period $n+1$, then 
    $$|g^{n+1}(0)| > 2^{-2}C^{-1} \ga^{-n}.$$
\end{lem}
    \begin{proof}
        Consider $r = 2^{-1} C^{-1}\ga^{-n}$.
        Then $|(g^{n+1})'| \leq 2^{-1}$ on $B(0,r)$.
        If $g^{n+1}$ does not have an attracting fixed point, $g^{n+1}(B(0,r)) \not\subset B(0,r)$ so $|g^{n+1}(0)| > r/2$, as required.  
    \end{proof}

    \begin{lem}\label{lemXiGan3m}
        Suppose 
        $h(a_0) < 2^{-m}\log 2$, so
        $f_{a_0}$ is $m$ times Feigenbaum renormalisable.
        Then 
        $$
        |\xi'_{2^m n -1}(a_0)| < \Ga^{n + 3m}.$$
    \end{lem}
    \begin{proof}
        Let $g_k$ denote the $k$th (Feigenbaum) renormalisation of $f_a$, omitting the dependence on $a$. Denote by $J_k$ the corresponding restrictive interval containing $0$. 
       Thus 
$$g_k := \left.f_{a}^{2^k}\right._{|J_k}.$$
Denote by $g_a$ the $m$th renormalisation of $f_a$ for $a$ in a neighbourhood of $a_0$. 
It is necessary to compute bounds for the derivative of the $n$th critical value of $g_a$ with respect to $a$ in a neighborhood of $a_0$. 
Observe that whenever $2^k | n$, $f_a^{n-1}(a)=f_a^{n}(0) \in J_k$, thus
$$|(f_a^{2^k})'(f_a^n(0))| = |g_k'(f_a^n(0))| \leq \Ga.$$
Decomposing the orbit of $f_a^j(a)$ according to visits to  $$J_l,\ldots,J_{m-1},J_m,J_m,\ldots,J_m,J_{m-1},\ldots,J_{k+1},J_k,$$ we obtain 
$$\frac{\ln |(f_a^{\hat n})'(f_a^j(a))|}{\ln \Ga} \leq \left\lfloor \frac{\hat n}{2^m}  \right\rfloor + 2m,$$
where $\lfloor y \rfloor$ is the integer part of $y \in \R$.
If we plug these estimates into~\eqref{equXiDer}, for any $n \geq 1$ we obtain
\begin{equation}\label{eqn:paaa}
    |\pa_a g_a^{n-1}(g_a(0))| = |\xi_{ 2^m n - 1 }'(a)| \leq 2^m \Ga^{2m}\sum_{i=0}^{n-1 }  \Ga^i  < \Ga^{n + 3m}.
\end{equation}
This is the desired bound. 
\end{proof}

\begin{lem} \label{lem:univrho}
        There exists $\rho>0$ such that,
        if $$\log \hat b \in (2^{-m-1}\log 2, 2^{-m}\log 2)$$ and $T_{\hat b}$ is periodic with period $2^m n$, then the length of the interval $h^{-1}(\hat b)$ is at least $\exp(-(n+m)\rho).$
    \end{lem}
    \begin{proof}
        Each element of $h^{-1}(\hat b)$ is a quadratic parameter which is $m$ times Feigenbaum renormalisable and then once renormalisable of period $n$. The left endpoint $a_C$ has a final renormalisation which is conjugate to Chebyshev; $h^{-1}(\hat b)$ contains a super-attracting parameter $a'$ of period $2^m n$, so $\xi_{2^m n -1}(a') = 0$. Let $k := 2^m n -1$. We will estimate $\xi_{k}(a_C)$, and $\xi'_{k}$ on $(a_C,a')$, giving a lower bound on $a'-a_C$. 

        If we set $J_0:=I_{a_C}$ (defined on page \pageref{pageIa}) and denote by  $$J_1 \supset J_2 \supset \cdots\supset J_m$$ the first $m$ restrictive intervals for $f_{a_C}$, it follows from \cite[Theorem~1]{Sul} that for some universal $\kappa>0$, $$|J_{k+1}|/|J_k| \geq e^{-\kappa}.$$
        In particular, the restrictive interval $J_m$ for the $m$th renormalisation $g = f_{a_C}^{2^m}$ has length $\delta  \geq \exp(-m \kappa)$. 
        Again by \cite[Theorem~1]{Sul}, 
        		$$|g'(x)| < Cx/\delta \leq C,$$
	for some universal constant $C$. 
        Since $g$ renormalises into Chebyshev, it does not have an attracting orbit. Applying Lemma~\ref{lem:przyt}, 
        $$|g^n(0)| > 2^{-2} \delta C^{-1} C^{-n+1} > \exp(-m\kappa - n\rho_1)$$
        for some universal $\rho_1>0$. 
        Reformulating, 
        $$
        |\xi_k(a_C)| > \exp(-m\kappa -n\rho_1).$$
        Meanwhile, from \rlem{XiGan3m},  $|\xi'_k| < \Ga^{n+3m}$. 
        Thus 
        $$|a'-a_C| > \exp(-m(\kappa +3\log \Ga) -n(\rho_1 + \log \Ga)).$$ 
    \end{proof}
    \begin{prop}
        There exist
 $C_1, \alpha_1$ for which, on each interval $[a_m, a_{m+1}]$, $h$ is $(C_1, \alpha_1)$-H\"older. 
 \end{prop}
 \begin{proof}
     Let $a,a' \in [a_m,a_{m+1}]$ have different entropies. Take $n$ minimal for which~\eqref{eq:guc1} holds, to obtain a periodic tent map $T_{\hat b}$ of period at most $2^m n$ and with $h^{-1}(\hat b) \subset [a,a']$. 
     By Lemma~\ref{lem:univrho}, $|a'-a| \geq \exp(-(n+m)\rho)$. 
     Since $n$ is minimal, 
     $$|h(a') - h(a)| < C_2 2^{-m}2^{-(n-1)/2}= \sqrt{2} C_2 \exp\left(-\left(m + \frac{n}2\right)\log 2\right).$$
     Taking $\alpha_1 = \frac{\log 2}{2\rho}$, 
     $$
     |h(a') - h(a)| < \sqrt{2}C_2 |a'-a|^{\alpha_1}.$$
 \end{proof}

 This completes the proof of \rthm{UniHol}.